\documentclass[12pt]{amsart}

\usepackage{amsmath}
\numberwithin{equation}{section}

\usepackage{subfigure}
    \newtheorem{thm}{Theorem}[section]
    \newtheorem{lem}{Lemma}[section]
    \newtheorem{prop}{Proposition}[section]
    \newtheorem{rem}{Remark}[section]

\usepackage[mathscr]{eucal}
\usepackage{amsmath,amssymb,amscd,amsthm}
\usepackage{color}
\usepackage{amsmath,amsthm,amssymb,graphicx}
\usepackage{epsfig,float}

\theoremstyle{definition}
\newtheorem{remark}{Remark}[section]
\theoremstyle{plane}
\newtheorem{example}{Example}[section]
\def \beq{ \begin{equation} }
\def \eeq{\end{equation}}
\def  \q {{\bf q}}
\def \p {{\bf p}}

\newcommand{\x}{ {\bf x} }
\newcommand{\y}{ {\bf y} }
\renewcommand{\v}{ {\bf v} }
\newcommand{\R}{ \mathbb{R} }
\newcommand{\M}{ \mathbb{M} }
\renewcommand{\S}{ \mathbb{S} }
\renewcommand{\H}{ \mathbb{H} }
\newcommand{\sn}{ \mathrm{sn\,} }
\newcommand{\csn}{ \mathrm{csn\,} }
\newcommand{\tn}{ \mathrm{tn\,} }
\newcommand{\ctn}{ \mathrm{ctn\,} }
\newcommand{\F}{ {\bf F} }
\newcommand{\e}{ {\bf e} }

\newcommand{\rg}{\rho_g}
\newcommand{\wa}{{\omega_\alpha}}
\newcommand{\wth}{{\omega_\theta}}

\usepackage[top=3cm, bottom=3cm, left=3cm, right=3cm]{geometry}

\title{The Vlasov-Poisson System for Stellar Dynamics in Spaces of Constant Curvature}
\begin{document}
\maketitle
\markboth{Florin Diacu, Slim Ibrahim, Crystal Lind, and Shengyi Shen}{The Gravitaional Vlasov-Poisson System in Curved Spaces}
\author{\begin{center}
{\bf Florin Diacu}$^{1,2}$, {\bf Slim Ibrahim}$^{1,2}$, {\bf Crystal Lind$^2$}, and {\bf Shengyi Shen$^2$}\\
\smallskip
{\footnotesize $^1$Pacific Institute for the Mathematical Sciences\\
and\\
$^2$Department of Mathematics and Statistics\\
University of Victoria\\
P.O.~Box 1700 STN CSC\\
Victoria, BC, Canada, V8W 2Y2\\
diacu@uvic.ca, ibrahims@uvic.ca, lind@uvic.ca, and shengyis@uvic.ca\\
}\end{center}

}

\vskip0.5cm

\begin{center}
\today
\end{center}

\begin{abstract}
We obtain a natural extension of the Vlasov-Poisson system for stellar dynamics to spaces of constant Gaussian curvature $\kappa\ne 0$: the unit sphere $\mathbb S^2$, for $\kappa>0$, and the unit hyperbolic sphere $\mathbb H^2$, for $\kappa<0$. These equations can be easily generalized to higher dimensions. When the particles move on a geodesic, the system reduces to a 1-dimensional problem that is more singular than the classical analogue of the Vlasov-Poisson system. In the analysis of this reduced model, we study the well-posedness of the problem and derive Penrose-type conditions for linear stability around homogeneous solutions in the sense of Landau damping. 
\end{abstract}

\section{Introduction}
The Vlasov-Poisson system models the density change of  galaxies in a cluster of galaxies, stars in a galaxy, or particles in plasma. The galaxies, stars, or particles are assumed to be identical to each other, while collisions, relativistic effects, and magnetic fields are neglected. If ignoring these effects is physically unreasonable in certain problems, some related models can be used, such as the Vlasov-Maxwell \cite{Andreasson11}, Einstein-Vlasov \cite{Rendall02}, or Vlasov-Manev systems \cite{Bobylev97}. We focus here on the case of stellar dynamics, assuming that collisions do not occur and relativistic effects can be neglected. Under these assumptions, the evolution of stars or galaxies is usually modelled in the framework of kinetic theory by the Vlasov-Poisson system in Euclidean space, given by the equations
\begin{equation*}
\begin{split}
\frac{\partial}{\partial t}f(t,\x,\v) + \v\frac{\partial}{\partial \x}f(t,\x,\v)+F(t,\x)\frac{\partial}{\partial \v}f(t,\x,\v)=0,\ \ \ \ \ \!  \\
F(t,\x)=-\int_{\mathbb{R}^3}\frac{\x-\y}{\mid \x-\y \mid^3}\rho(t,\y)d\y,\ \ \rho(t,\x):=\int_{\mathbb{R}^3} f(t,\x,\v)d\v,
\end{split}
\end{equation*}
where $U$ represents the potential. For an initial value problem, a value $f(0,\x,\v)$ is given to determine the distribution function $f(t,\x,\v)$ of the celestial objects at position $\x\in \mathbb{R}^3$ with velocity $\v\in \mathbb{R}^3$ at time $t\in \mathbb{R}$. Due to the dependence of the potential on the density $\rho$, the system reduces to a non-linear partial differential equation, which is difficult to understand. Even the global existence of solutions remained an open problem for decades. The first to attack it was R.\ Kurth, who showed the local existence of solutions in 1952 \cite{Kurth52}. In 1977, J.\ Batt proved global existence for spherically symmetric solutions \cite{Batt77} and, in 1985, C.\ Bardos and P.\ Degond completed the global existence of solutions under the assumption of small initial data \cite{Bardos85}. Finally, the existence of global solutions with general initial data was achieved by K.\ Pfaffelmoser in 1990 \cite{Pfaffelmoser92} and, independently, by P.L.\ Lions and B.\ Perthame in 1991 \cite{Lions91}. Recently, C.\ Mouhot and C.\ Villani showed that these equations exhibit Landau damping, a remarkable result that inspired a direction of research we took here, \cite{MV}.

In this paper we broaden the scope of the Vlasov-Poisson system to spaces of non-zero constant Gaussian curvature
within the framework of classical mechanics, without involving special or general relativity. On small scales, the curvature of the physical space is negligible, but we cannot exclude the possibility that the universe is hyperbolic or elliptic on the large scale. By extending the study of the Vlasov-Poisson system to spaces of non-zero constant curvature we could, on one hand, obtain a better understanding of the flat case by viewing the Vlasov-Poisson system as the limit of its counterpart in curved space when the curvature tends to zero. On the other hand, we might be able to decide whether the universe is curved or not. Indeed, if a certain solution of the density function occurs only in, say, flat space but not in hyperbolic and elliptic space, and such behaviour is supported by astronomical evidence, then we could claim that space is Euclidean. Even if only for these two reasons alone, the extension of the Vlasov-Poisson system to spaces of constant curvature deserves a detailed study.

The derivation of the classical system for stellar dynamics uses the Newtonian equations of the gravitational $N$-body problem. To generalize the Vlasov-Poisson system to elliptic and hyperbolic space, we employ a meaningful extension of the Newtonian $N$-body problem to spaces of constant Gaussian curvature. Although the idea of such an extension belonged to Bolyai and Lobachevsky in the 2-body case \cite{Loba, Bolyai}, a suitable generalization of the classical Newtonian system was only recently obtained, \cite{DiacuBook, DiacuMemoirs}. 
Suitability is meant here in a mathematical sense, since there are no physical ways of testing the new system. More precisely, the potential of the ``curved problem'' in its most simple setting (the case of one body moving around a fixed attractive centre---the so-called Kepler problem), not only recovers the Newtonian case as the curvature tends to zero, but also inherits two properties of the classical potential: (i) it is a harmonic function, i.e.\ satisfies Laplace's equation; (ii) all of its bounded orbits are closed. Moreover, the properties that have been so far discovered for the extension of the $N$-body problem to spaces of constant curvature fit nicely with what is known in the classical case \cite{Diacu11, Diacu12, Diacu13, DiacuI, DiacuII}. 

Our derivation of the Vlasov-Poisson system is based on these equations of the curved $N$-body problem and on Liouville's theorem. We could also regard these equations in a mean field approximation to obtain, at least formally, the Vlasov equation in spaces of constant curvature. Both these approaches use coordinate systems (extrinsic or intrinsic). An approach in the spirit of geometric mechanics to define the Vlasov equation, without the use of coordinates, is due to  J.\ Marsden and A.\ Weinstein \cite{Marsden85}. This method regards the Vlasov equation as a conservation of a distribution function along phase-space trajectories of a given Hamiltonian. When the potential in the Hamiltonian is given by Poisson's equation, we naturally obtain a Vlasov-Poisson system on symplectic manifolds and, in particular, on spaces of constant curvature. The analysis of such systems on manifolds appears to be difficult since the transport part does not seem to be well understood. For this reason, at the current stage of our research, we focus on the special case when the particles move along a geodesic of the manifold. As geodesics are invariants of the equations of motion, such configurations are also invariant. By imposing this restriction, we derive a new (1-dimensional) reduced model given by system \eqref{reduced} in Proposition \ref{1DS2}, which provides our first notable result. 

This model is different from the classical 1-dimensional Vlasov-Poisson system \cite{Bardos2013} and the Vlasov-Hamilton Mean Field equations \cite{Faou14} because the interaction potential is more singular. However, the mathematical challenges are similar to those of the Vlasov equation given by the Dirac potential (also called gyrokinetic), derived in \cite{Ghendrih09}, \cite{Grenier95}, \cite{Jin99}, which is important in the analysis of plasma fusion. Given the singularity in the potential, we focus (as in \cite{analyticsol}) on the existence of strong solutions, which we construct in the class of Gevrey functions. This is our second notable result, and we state it in Theorem \ref{mainthem}.

Following C.\ Villani \cite{VillaniLectures}, we then derive Penrose-type conditions for our 1-dimensional Vlasov-Poisson system such that the homogeneous states are linearly stable, i.e.\ require convergence for density and force. The convergence (or decay) rate is exponential when the model is reduced from the sphere. We want to emphasize that this decay becomes only algebraic for the reduced hyperbolic system. This phenomenon occurs because of the low frequencies when we work on the whole real line. In this context, our final notable result, which states that the linearized Vlasov-Poisson system along a geodesic has stable solutions in the sense that they exhibit the Landau damping phenomenon, is stated in Theorems \ref{thm_S2} and \ref{thm_H2}, the former dealing with the sphere and the latter treating the hyperbolic case.

Our paper is organized as follows. In Section 2, we lay the background for further developments. After introducing some notation, we provide a solution representation to Poisson's equation based on the solution of the Laplace-Beltrami equation obtained in \cite{Cohl11Sphere}, for the sphere, and \cite{Cohl12}, for the hyperbolic sphere. In Section 3, we derive the Vlasov-Poisson system on 2-dimensional curved spaces (the unit 2-sphere and the unit hyperbolic 2-sphere), using the extension of the Newtonian equations mentioned above, and write this system in Hamiltonian form. The generalization of these equations to higher dimensions is straightforward. Then we 
obtain an expression of the gravitational field generated by the bodies (assumed to be point particles), which we apply to a 1-body example to illustrate the use of symmetry and the agreement between our original assumptions and our derived equations. For simplicity, in Section 4 we choose initial data such that the particles move on a geodesic of the sphere or hyperbolic sphere. Since the geodesics are invariant sets for the equations of motion, this approach allows us to use the symmetry of the configuration to simplify the system. We continue under this assumption for the remainder of the paper, and in Section 5 we prove the existence of solutions to our Vlasov-Poisson system. Finally, in Section 6, we state and prove our stability results mentioned above. 

\section{Background}

Since we are interested only in qualitative properties of solutions, we will choose from the start the physical units such that the gravitational constant is 1. In the discrete case, when  dealing with a finite number of point masses, the qualitative study of the motion can be reduced to the unit sphere, for positive curvature, and the unit hyperbolic sphere, for negative curvature, \cite{DiacuBook, DiacuMemoirs}. Since this framework can be used without loss of generality in our qualitative studies as well, we will also employ it here. 

Let $\mathbb M^2$ represent the unit 2-sphere,
$$
\mathbb S^2=\{(x_1,x_2,x_3)\ \! | \ \! x_1^2+x_2^2+x_3^2=1\},
$$
for positive curvature, and the unit hyperbolic 2-sphere (i.e.\ the upper sheet of the unit hyperboloid of two sheets),
$$
\mathbb H^2=\{(x_1,x_2,x_3)\ \! | \ \! x_1^2+x_2^2-x_3^2=-1,\ x_3>0\},
$$
for negative curvature, where $\S^2$ is embedded in the 3-dimensional Euclidean space $\R^3$ and $\H^2$ is embedded in the 3-dimensional Minkowski space $\R^{2,1}$, in which all coordinates are spatial. The origin of the coordinate system lies at the centre of the sphere and the hyperbolic sphere. If $\kappa$ denotes the Gaussian curvature, the signum function is defined as
$$
\sigma=
\begin{cases}
+1, \ \ {\rm for}\ \ \kappa\ge 0\cr
-1,\ \ {\rm for} \ \ \kappa< 0. \cr
\end{cases}
$$
Following \cite{DiacuBook} and \cite{DiacuMemoirs}, we introduce unified trigonometric functions,
\begin{equation}
\begin{array}{lll}
\sn x := 
\begin{cases} 
\sin x,\ \ \ \  {\rm for} \ \    \kappa> 0\crcr
\sinh x, \ \ \ \! {\rm for}\ \ \kappa<0,
\end{cases} 
&& 
\csn x := 
\begin{cases} 
\cos x, \ \ \ \   {\rm for} \ \ \kappa> 0\crcr
\cosh x, \ \ \ \! {\rm for} \ \ \kappa<0, 
\end{cases}\\[10mm]
\tn x := \dfrac{\sn x}{\csn x}, 
&& 
\ctn x := \dfrac{\csn x}{\sn x},
\end{array}
\end{equation}
in order to treat both $\S^2$ and $\H^2$ simultaneously. In this framework, we can write the geodesic distance between the points ${\bf a}=(a_1,a_2,a_3)$ and ${\bf b}=(b_1,b_2, b_3)$ of $\mathbb M^2$ as
\begin{equation}
\label{distance}
d({\bf a},{\bf b})=
\csn^{-1}(\sigma{\bf a}\cdot{\bf b}),
\end{equation}
where the operation $\cdot$
stands for the scalar product
$$
{\bf a}\cdot{\bf b}=a_1b_1+a_2b_2+\sigma a_3b_3,
$$ 
which is inherited from the space in which the manifold is embedded, i.e.\ $\R^3$ for $\S^2$, but $\R^{2,1}$ for $\H^2$. We define the norm of the vector $\bf a$ as
$$
||{\bf a}||=|a_1^2+a_2^2+\sigma a_3^2|^{1/2}.
$$
The gradient of a scalar function $\mathfrak f=\mathfrak f(\x)$, with $\x\in \M^2$, is given by
\begin{equation}
\label{grad_ext}
\nabla_{\bf x}\mathfrak f=\left(\partial_{x_1}\mathfrak f\right) \e_1 + \left(\partial_{x_2}\mathfrak f\right) \e_2 + \left(\sigma\partial_{x_3}\mathfrak f\right) \e_3,
\end{equation}
where $\e_1, \e_2, \e_3$ denote the elements of the canonical base in Cartesian coordinates. The divergence of a vector function $\F=F_1(\x)\, \e_1 + F_2(\x)\, \e_2 + F_3(\x)\, \e_3$, where $\x:=x_1\e_1 + x_2\e_2 + x_3\e_3 \in\M^2$, has the form
\begin{equation}
\label{div_ext}
\mathrm{div}_{\x}\;{\bf F}= \partial_{x_1}F_1+ \partial_{x_2}F_2+ \partial_{x_3}F_3.
\end{equation}
The Laplace-Beltrami operator on $\M^2$, applied to a scalar function $\mathfrak f$ as defined above, can then be written as
\begin{equation}
\label{laplace_ext}
\Delta_{\x}\mathfrak f=\partial_{x_1 x_1} \mathfrak f + \partial_{x_2 x_2} \mathfrak f + \sigma\partial_{x_3 x_3} \mathfrak f.
\end{equation}
Notice that the operators $\nabla_{\bf x}, \mathrm{div}_{\bf x}$, and $\Delta_{\bf x}$ act on functions defined on the manifold $\mathbb M^2$ after these functions are extended to the 3-dimensional ambient space by replacing the variable $\bf x$ with ${\bf x}/||{\bf x}||$. Once the computations are performed, the resulting functions can be restricted to the manifold again by imposing the condition $||{\bf x}||=1$, which defines $\mathbb M^2$.

Recall that the tangent space at $\x\in\M^2$ is given by
$$
T_{\x}\M^2 = \left\{\v \mid \x\cdot \v = 0\right\}
$$
and the tangent bundle $T\M^2$ is the union of all tangent spaces, i.e.
$$
T\M^2 =\bigcup_{\x\in\mathbb M^2}T_{\x}\M^2=\left\{(\x,\v)\mid \x\in\M^2,\ \x\cdot\v = 0\right\}.
$$

To parametrize $\mathbb M^2$, let $I_{\M^2}$ be the interval  $[0,\pi]$ if $\M^2=\S^2$, but $[0,\infty)$ if $\M^2=\H^2$. Then the position of a particle at $\x\in\M^2$ can be represented as
\begin{equation}\label{parametrize}
\x=x_1 \e_1 + x_2 \e_2 + x_3 \e_3 = (\sn\alpha \cos\theta) \e_1 + (\sn\alpha\sin\theta)\e_2 + (\csn\alpha)\e_3,
\end{equation}
where $\theta\in[0,2\pi)$ and $\alpha\in I_{\M^2}$.
We define $\e_r$, $\e_\alpha$, and $\e_\theta$ to be the orthogonal vectors that satisfy the equation
\begin{equation}
\label{unit_vec}
\begin{bmatrix}
\e_r\\
\e_\alpha\\
\e_\theta
\end{bmatrix}
=
\begin{bmatrix}
\sn\alpha\cos\theta & \sn\alpha\sin\theta & \csn\alpha\\
\csn\alpha\cos\theta & \csn\alpha\sin\theta & - \sigma\sn\alpha\\
-\sin\theta & \cos\theta & 0 
\end{bmatrix}
\begin{bmatrix}
\e_1\\
\e_2\\
\e_3
\end{bmatrix}.
\end{equation}
When the parameters $\alpha$ and $\theta$ depend on time, the velocity of a particle at $\bf x$ is 
$$\v=\omega_\alpha\,\e_\alpha + \left(\omega_\theta\,\sn\alpha\right)\e_\theta,$$
where $\omega_\alpha:=\dot{\alpha}$ and $\omega_\theta:=\dot{\theta}$, and the upper dot denotes the derivative relative to time. The acceleration is then given by
\begin{equation*}
\begin{array}{rcl}
\mathbf{a}=-\sigma\,(\omega_\alpha^2+\,\omega_\theta^2\,\sn^2\alpha)\e_r + (\dot{\omega}_\alpha - \omega_\theta^2\,\sn\alpha\,\csn\alpha)\e_\alpha + (\dot{\omega}_\theta\sn\alpha + 2 \omega_\alpha\,\omega_\theta\,\csn\alpha)\e_\theta.
\end{array}
\end{equation*}
The natural Riemannian or pseudo-Riemannian metric on $\M^2$ has the form 
\begin{equation*}\label{metric_tensor}
[g_{ij}]_{i,j=1,2}= \begin{bmatrix}
1 & 0\\
0 & \sn^2\alpha
\end{bmatrix},
\end{equation*}
with inverse
\begin{equation*}\label{metric_tensor_inv}
[g^{ij}]_{i,j=1,2}= \begin{bmatrix}
1 & 0\\
0 & \dfrac{1}{\sn^{2}\alpha}
\end{bmatrix}.
\end{equation*}
Then, in the local coordinates ${\bf q}:=\alpha{\bf e}_\alpha+\theta{\bf e}_\theta$, the gradient of $\tilde{\mathfrak f}=\tilde{\mathfrak f}(\alpha,\theta):=\mathfrak f(\x(\q))$ at ${\bf q}\in\M^2$ is given by
\begin{equation}
\label{grad_local}
\nabla_{\bf q}\ \! \tilde{\mathfrak f}
= (\partial_\alpha\tilde{\mathfrak  f})\e_\alpha + \left(\dfrac{1}{\sn\alpha}\,\partial_\theta\tilde{\mathfrak  f}\right)\e_\theta,
\end{equation}
the Laplace-Beltrami operator takes the form
\begin{equation*}
\label{laplace_local}
\begin{array}{lll}
\Delta_{\bf q}\ \! \tilde{\mathfrak f} = \mathrm{div}_{\bf q}(\nabla_{\bf q} \tilde{\mathfrak f}) = \ctn\alpha\ \!\partial_\alpha \tilde{\mathfrak  f} + 
\partial_{\alpha\alpha}\tilde{\mathfrak f}+ \dfrac{1}{\sn\alpha}\partial_{\theta\theta}\tilde{\mathfrak f},
\end{array}
\end{equation*}
where the divergence of a vector function
$$
\tilde{\bf F}=\tilde{\bf F}(\alpha,\theta):={\bf F}(\x(\alpha,\theta))= \tilde{F}_{\alpha}(\alpha,\theta)\e_\alpha + \tilde{F}_{\theta}(\alpha,\theta)\e_\theta
$$
at ${\bf q}\in\M^2$ is expressed as 
\begin{equation*}
\label{div_local}
\begin{array}{lll}
\mathrm{div}_{\bf q}\ \! \tilde{\bf F} = 
\tilde{F}_\alpha\ \!\ctn\alpha+\partial_\alpha\tilde{F}_\alpha+\dfrac{1}{\sn\alpha}\ \!\partial_{\theta}\tilde{F}_\theta.
\end{array}
\end{equation*}
The volume form on $\M^2$ is 
\begin{equation*}
\label{volume_form}
\Omega = \sn\alpha\; d\alpha d\theta.
\end{equation*}

\section{Equations of motion, gravitational potential and the curved Vlasov-Poisson system}

In this section we first derive the equations of motion for a particle moving on the manifold $\mathbb M^2$, then introduce the gravitational potential and the gravitational force function, and finally obtain the Vlasov-Poisson system on $\mathbb S^2$ and $\mathbb H^2$. Although, to fix the ideas, we work here only in the 2-dimensional case, our equations can be extended to any finite dimension. 
\subsection{Equations of motion}
Recall that for $(\x,\v)\in T\M^2$, 
$f=f(t,\x,\v)$ denotes the phase space density, $\rho=\rho(t,\x)$ the spatial density, and $U=U(t,\x)$ the gravitational potential function.

\begin{prop}
In extrinsic coordinates having the origin at the centre of $\mathbb M^2$, the equations of motion for a particle with position $\x\in\M^2$ and velocity $\v=v_1\e_1 + v_2\e_2 + v_3\e_3\in T\M^2$, under the effect of a potential function $U:\mathbb R\times\M^2\rightarrow \R$, are
\begin{equation}\label{eom_ext}
\begin{cases}
\dot{\x} = \v,\cr
\dot{\v} = \nabla_{\x} U(t,\x) - \sigma(\v \cdot \v)\x.
\end{cases}
\end{equation}
\end{prop}
\begin{proof}
The proof of this result can be found in Section 3.4 of \cite{DiacuBook} for a finite number of masses for the Newtonian potential function. We will adapt it here to the form of the potential function $U$ given in \eqref{potential_ext} in the case of one body. We need to use first a variational result from constrained Lagrangian mechanics (see, e.g., \cite{Fowles}), since the motion of the particle is restricted to $\mathbb M^2$. Under such circumstances, let $L$ represent the Lagrangian and $g({\bf x})=0$ be the equation expressing the constraint, i.e. 
\begin{equation}\label{Lagrangian}
L(t,{\bf x},{\bf v})=T({\bf x},{\bf v})-V(t,{\bf x})=\dfrac{1}{2}\sigma (\v \cdot \v)(\x \cdot \x)+U(t,{\bf x}),
\end{equation}
and
\begin{equation}\label{constraint}
g({\bf x})=x_1^2 +x_2^2+\sigma x_3^2-\sigma=0.
\end{equation}
The factor $\sigma(\x \cdot \x)=1$ is introduced to allow a Hamiltonian representation for the equations of motion (see Section 3.6 of \cite{DiacuBook}).
Then, the equations of motion are given by the Euler-Lagrange system with constraints,
\begin{equation}\label{eom_formula_constraint}
\dfrac{d}{dt}(\partial_{\bf v}L(t,{\bf x},{\bf v}))-\partial_{\bf x} L(t,{\bf x},{\bf v})-  \lambda\ \! \partial_{\bf x} g({\bf x})=0,
\end{equation}
where $\lambda$ is the Lagrange multiplier. After substitution, equation \eqref{Lagrangian} simplifies to
\begin{equation}
\label{eom2}
\sigma\dot{\v}(\x \cdot \x) - \sigma(\v \cdot \v)\x - \nabla_{\x} U(t,{\bf x}) - 2\lambda \x = 0,
\end{equation}
where $\lambda$ is unknown and we have used the fact that $\x\cdot\v = 0$ for $\x\in \M^2$. To determine $\lambda$, we take first the scalar product of $\x$ with the left hand side in \eqref{eom2} and obtain
\begin{equation}
\label{eom3}
\sigma(\x\cdot \dot{\v})(\x \cdot \x) - (\v \cdot \v)(\x \cdot \x) - \x\cdot \nabla_{\x} U(t,{\bf x}) - 2\lambda(\x \cdot \x) =0.
\end{equation}
Since the particle is constrained to $\M^2$, we can differentiate equation \eqref{constraint} twice with respect to time to get 
$$
\v \cdot \v + \x\cdot \dot{\v}=0.
$$ 
Using this equation together with Euler's formula \eqref{Euler's}, we can simplify \eqref{eom3} to
\begin{equation*}
-(\v \cdot \v)(\x \cdot \x) - (\v \cdot \v)(\x \cdot \x) - 2\lambda(\x \cdot \x) = 0,
\end{equation*}
from which we get $\lambda=-(\v \cdot \v)$. Substituting this value of $\lambda$ and $\x \cdot \x=\sigma$ into \eqref{eom2} leads us to the equation
\begin{equation*}
\begin{array}{lll}
\dot{\v}= \nabla_{\x}  U(\x) - \sigma(\v \cdot \v)\x,
\end{array}
\end{equation*}
a remark that completes the proof.
\end{proof}
The following proposition gives the equations of motion in local coordinates.
\begin{prop}
In local coordinates, the equations of motion for a particle with position $\q=\alpha\e_\alpha+\theta\e_\theta\in\M^2$ and velocity $\p=\omega_{\alpha}\e_\alpha + \sn\alpha\,\omega_{\theta}\e_\theta\in T\mathbb M^2$, under the effect of a potential function $\tilde{U}:\mathbb R\times\M^2\rightarrow \R$ are given by the system
\begin{equation}\label{eom_local}
\begin{cases}
\dot{\alpha} = \omega_\alpha,\cr
\dot{\theta} = \omega_\theta, \cr
\dot{\omega}_\alpha = \partial_\alpha\tilde{U}(t,\alpha,\theta) + \omega_\theta^2\,\sn\alpha\,\csn\alpha,\cr
\dot{\omega}_\theta = \dfrac{1}{\sn^2\alpha}\partial_\theta\tilde{U}(t,\alpha,\theta)-2\omega_\alpha\omega_\theta\,\ctn{\alpha}.
\end{cases}
\end{equation} 
\end{prop}
\begin{proof}
The Lagrangian for our system is given by the difference between kinetic and potential energies, i.e. 
\begin{equation}\label{lagrangian}
\tilde{L}(t,\alpha,\theta,\dot\alpha,\dot\theta) = \dfrac{1}{2} (\dot{\alpha}^2 + \dot{\theta}^2\sn^2\alpha) + \tilde{U}(t,\alpha,\theta).
\end{equation}
Substituting this into the Euler-Lagrange equations,
\begin{equation}\label{eom_lagrangian}
\begin{array}{lll}
\dfrac{d}{dt} (\partial_{\dot\alpha}\tilde{L}) - \partial_\alpha\tilde{L} = 0 &\mathrm{and}& \dfrac{d}{dt} (\partial_{\dot\theta}\tilde{L}) - \partial_\theta\tilde{L} = 0,\\[5mm]
\end{array}
\end{equation}
yields the desired equations.
\end{proof}
The following obvious result expresses the equations of motion in the context of the Hamiltonian formalism. 
\begin{remark}
In extrinsic coordinates $(\x,\v)\in T\M^2$, the equations of motion for a particle of mass 1 moving on $\M^2$ can be written in Hamiltonian form as
\begin{equation*}
\begin{cases}
\dot{\x} = \partial_\v H\cr
 \dot\v = -\partial_\x H, 
\end{cases}
\end{equation*}
where
$H(t,\x,\v)=\dfrac{1}{2}\sigma(\v\cdot\v)(\x\cdot\x)-U(t,\x)$ is the Hamiltonian function.
\end{remark}

\subsection{Gravitational potential and gravitational field}
To define the Vlasov-Poisson system on $\mathbb M^2$, we need to obtain a suitable solution representation to Poisson's equation on $\mathbb S^2$ and $\mathbb H^2$. This representation is provided below in terms of the potential function $U$. 
\begin{prop}\label{potential_ext_prop}
The gravitational potential function $U$ at a point $\x\in\M^2$ due to a spatial mass distribution $\rho=\rho(t,\x)$ is given by
\begin{equation}\label{potential_ext}
U(t,\x) = \dfrac{1}{4\pi} \displaystyle\iint_{\M^2} \rho(t,\y)\log\left(\dfrac{1+\sigma \x\cdot\y }{\sigma- \x\cdot \y}\right)d\y.
\end{equation}
\end{prop}

\begin{proof}
The proof is trivial given that the fundamental solution of the Laplace-Beltrami operator on the manifold $\mathbb M^2$ is given by (see for example \cite{Cohl11Sphere, Cohl12})
\begin{equation}
\label{green's}
\mathcal{G}(\x)=\dfrac{1}{2\pi} \log\ctn\left[\dfrac{d(\x,\y)}{2}\right],
\end{equation}
where $d$ is defined as in \eqref{distance}. In particular, it satisfies
\begin{equation}
-\Delta_{\x} \mathcal{G}(\x) = \dfrac{\delta(\alpha-\alpha')\otimes\delta(\theta-\theta')}{\sn\alpha'},
\end{equation}
where $\delta$ is the usual Dirac delta distribution and $\otimes$ denotes the tensor product. Thus
\begin{equation}
U(\x) = \iint_{\M^2}\mathcal{G}(\x)\rho(\y)d\y.
\end{equation}
Now, since $ d(\x,\y)=\csn^{-1}(\sigma\x\cdot\y)$, and thanks to the trigonometric identity
$$
\ctn\frac d2=\sigma\frac{1+\csn d}{1-\csn d},
$$
the potential simplifies to 
\begin{equation}
U(\x) = \dfrac{1}{4\pi}\iint_{\M^2}\log\left[\dfrac{1+\sigma \x\cdot \y}{\sigma-\x\cdot \y}\right]\rho(\y)dy,
\end{equation}
as desired.
\end{proof}
Observe that if we extend the potential $U$ as 
\begin{equation}
\mathring{U}({\x}) = U\left(\dfrac{\x}{\|\x\|}\right) = \dfrac{1}{2\pi} \iint_{\M^2} \rho(\y) \log \left(\dfrac{\|\x\| + \sigma \x\cdot\y}{\sigma\|\x\| - \x\cdot \y}\right) d\y, 
\end{equation}
then $\mathring U$ becomes homogeneous of degree 0, and Euler's formula yields
the identity
\begin{equation}\label{Euler's} 
\mbox{$\x\cdot\nabla_{\x} \mathring{U}(t,{\bf x})=0$}.
\end{equation}
The physical interpretation of \eqref{Euler's} is that, in an extrinsic coordinate system having the origin at the centre of $\mathbb M^2$, the gravitational force acting on each particle is always orthogonal to its position vector $\x$. Therefore we can conclude that if particles are initially located on $\M^2$ with velocities tangent to the manifold, then they will remain on $\M^2$ for all time. The proof of this fact, obtained in Section 3.7 of \cite{DiacuBook} for the discrete Newtonian case, can be easily extended to our problem.

\begin{prop}
The gravitational field present at position $\x\in\M^2$ is given by
\begin{equation}
F(t,x):=\nabla_{\x} U(t,\x) =\dfrac{1}{2\pi\sigma} \displaystyle\iint_{\M^2} \dfrac{\y-\sigma(\x\cdot \y)\x}{1-(\x\cdot \y)^2}\rho(\y)d\y.
\end{equation}
\end{prop}
\begin{proof}
We begin with \eqref{potential_ext} and extend the expression by homogeneity so that it becomes
\begin{equation}
U(t,\x) = \dfrac{1}{4\pi} \displaystyle\iint_{\M^2} \rho(t,\y)\log\left(\dfrac{\|\x\|+\sigma \x\cdot\y }{\sigma\|\x\|- \x\cdot \y}\right)d\y.
\end{equation}
From here, we use the gradient as in \eqref{grad_ext} to find
\begin{equation}
\nabla_{\x} U(t,\x) = \dfrac{1}{2\pi\sigma} \displaystyle\iint_{\M^2} \dfrac{\|\x\|\y-\sigma\|\x\|^{-1}(\x\cdot \y)\x}{\|\x\|^2-(\x\cdot \y)^2}\rho(\y)d\y.
\end{equation}
Invoking $\|\x\| = 1$ yields the required result.
\end{proof}

If we consider a specific mass distribution, it is also possible to use Gauss's law to calculate the gravitational field without first knowing the gravitational force function. In fact this method proves to be useful in the case of symmetric mass distributions, as shown in the following example.
\begin{example}\label{example_point}
{\rm
Let a mass distribution on $\S^2$ be given by
\begin{equation*}
\rho(\alpha, \theta):=\rho(t,\alpha, \theta)=\dfrac{\delta(\alpha)\otimes\delta(\theta)}{\sin\alpha},
\end{equation*}
i.e. a point mass located at the north pole of the sphere, as shown in Figure \ref{point_part}.
\begin{figure}[!h]
\centering
\includegraphics[width=2.5in]{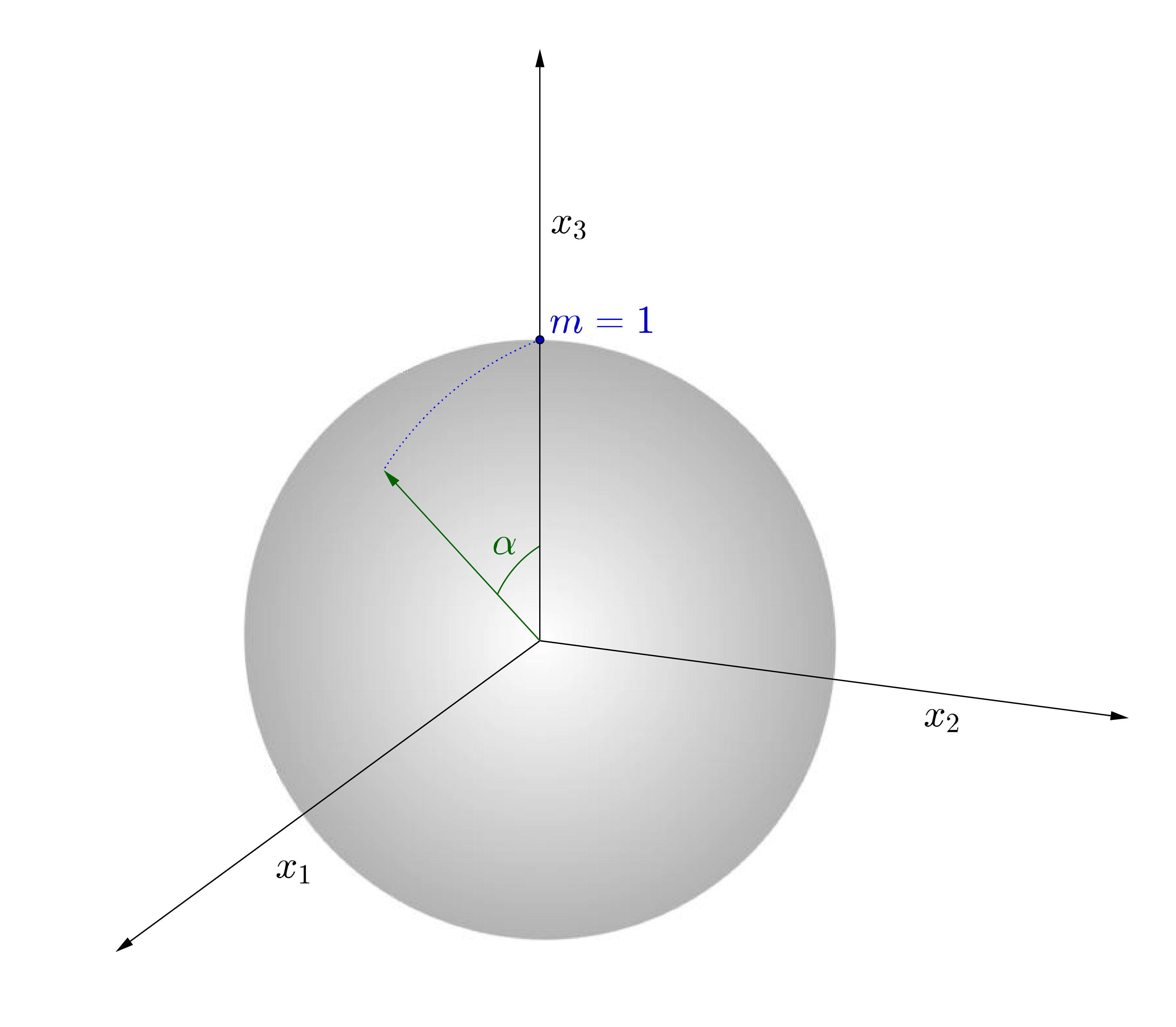}
\caption{Illustration of a point mass located at the north pole of $\S^2$.}
\label{point_part}
\end{figure}
Gauss's law in two dimensions says that the gravitational flux at any point $q\in\S^2$ is proportional to the mass enclosed by a Gaussian curve passing through the point, i.e. 
\begin{equation}
\label{Gauss}
\displaystyle\int_C \mathbf{F}\cdot \mathbf{n}= m,
\end{equation}
where $\mathbf{F}$ denotes the gravitational field, $\mathbf n\in T_\mathbf q\S^2$ is the normal vector to the curve $C$, and $m$ is the enclosed mass.

\begin{figure}[h!]
\centering
\includegraphics[width=2.5in]{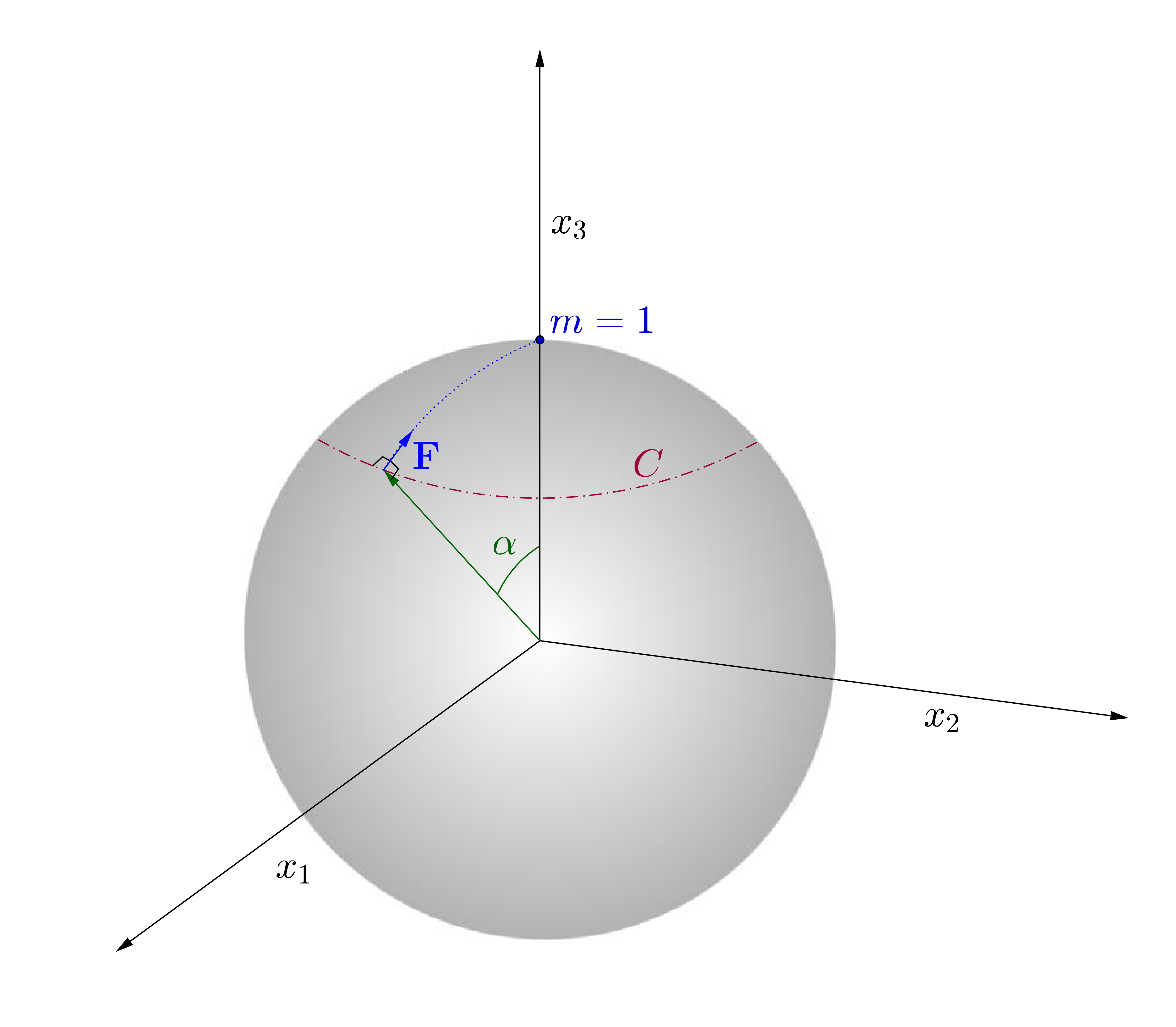}
\caption{A Gaussian curve, $C$, chosen for Example \ref{example_point}.}
\label{point_part_gauss}
\end{figure}

If we choose a Gaussian curve $C$ as pictured in Figure \ref{point_part_gauss}, the spherical symmetry of the mass results in significant simplification of \eqref{Gauss}. Since the field generated by the mass intersects the Gaussian curve perpendicularly within $\S^2$, and the magnitude of the field is the same at each point on the curve, our equation reduces to
\begin{equation*}
2\pi\sin\alpha|\mathbf{F}|=1.
\end{equation*}
Solving for $\mathbf{F}$ yields
\begin{equation}\label{point_field}
\mathbf{F}=-\dfrac{1}{2\pi\sin\alpha}\bf{e}_{\alpha}.
\end{equation}
Alternatively, we can calculate the same field by taking the gradient of the gravitational force function for the point mass. To do so, we first find the force function using \eqref{potential_ext} with \eqref{volume_form},
$$
\begin{array}{lll}
\tilde{U}(t,\alpha,\theta) &=& \dfrac{1}{4\pi} \displaystyle\int_{0}^{2\pi} \displaystyle\int_{0}^{\pi} \dfrac{\delta(\alpha')\otimes\delta(\theta')}{\sin\alpha'} \log\left[\dfrac{1+ \x(\alpha, \theta)\cdot\y(\alpha',\theta')}{1-\x(\alpha, \theta)\cdot\y(\alpha',\theta')}\right]\sin\alpha'd\alpha'd\theta'\\[5mm]
&=&  \dfrac{1}{4\pi} \log\left[\dfrac{1+\x(\alpha, \theta)\cdot\y(0,0)}{1-\x(\alpha, \theta)\cdot\y(0,0)}\right]
=\dfrac{1}{4\pi} \log\left(\dfrac{1+\cos\alpha}{1-\cos\alpha}\right)\\[5mm]
&=& \dfrac{1}{2\pi} \log\left[\cot\left(\dfrac{\alpha}{2}\right)\right],\\[5mm]
\end{array}
$$
then use \eqref{grad_local} to take the gradient and obtain
$$
\nabla_{\bf q} \tilde{U}(t,\alpha,\theta) = \dfrac{1}{2\pi} 
\partial_\alpha\!\left(\log\cot\dfrac{\alpha}{2}\right)\!{\bf e}_{\alpha}= -\dfrac{1}{2\pi\sin\alpha}\ \! { \bf e}_{\alpha},\\[5mm]
$$
which agrees with \eqref{point_field}.}
\end{example}
\subsection{The Vlasov-Poisson system on spaces of constant curvature}
According to kinetic theory and Liouville's theorem (see, e.g., \cite{Liboff}), the equation that governs the motion of a continuous particle distribution with no collisions is given in local coordinates by
\begin{equation*}
\frac{d}{dt}\tilde{f}(t,\alpha,\theta,\omega_\alpha,\omega_\theta)=0,
\end{equation*}
where $\tilde{f}=\tilde{f}(t,\alpha,\theta,\omega_\alpha,\omega_\theta):=f(t,\x(\alpha,\theta),{\bf v}(\omega_\alpha,\omega_\theta))$ and $f$ is the phase-space distribution function in extrinsic coordinates. Using the chain rule, this equation becomes
\begin{equation*}
\partial_t \tilde{f} + \dot{\alpha}\;\partial_{\alpha}\tilde{f} + \dot{\theta}\;\partial_{\theta}\tilde{f} + \dot{\omega}_\alpha\;\partial_{\omega_\alpha}\tilde{f} + \dot{\omega}_\theta\;\partial_{\omega_\theta}\tilde{f} =0.
\end{equation*}
Employing the equations of motion \eqref{eom_local}, we can write the above equation as
\begin{equation}
\begin{array}{lll}
0=&&\partial_t \tilde{f} + \omega_\alpha \partial_\alpha \tilde{f}+ \omega_\theta \partial_\theta \tilde{f}+ (\partial_\alpha \tilde{U} + \omega_{\theta}^2 \sn\alpha\csn\alpha) \partial_{\omega_\alpha} \tilde{f}\\[4mm]
&&+ \left(\dfrac{1}{\sn^2\alpha} \partial_\theta \tilde{U}- 2\omega_\alpha\omega_\theta\ctn\alpha\right) \partial_{ \omega_\theta} \tilde{f},
\end{array}
\end{equation}
which we call the Vlasov equation on $\M^2$ or the curved Vlasov equation. Using the same approach in extrinsic coordinates yields the equation
\begin{equation}
\label{vlasov_ext}
\partial_t f(t,\x,\v) + \v \cdot \nabla_{\x} f(t,\x,\v) + \left[\nabla_{\x}U(t,\x) - \sigma (\v \cdot \v) \x \right]\cdot \nabla_\v f(t,\x,\v)=0,
\end{equation}
where $\nabla_{\v}f = \partial_{v_1}f\,\e_1 +\partial_{v_2}f\,\e_2 + \sigma\partial_{v_3}f\,\e_3$, and we require that $\x \cdot \x=\sigma$ and $\x \cdot \v=0$, such that the particles remain on $\M^2$ during the motion. When we couple the Vlasov equation to our solution representation of Poisson's equation and add a compatibility condition between our spatial density and phase-space density, the result is a closed system, which we call the gravitational Vlasov-Poisson system in spaces of constant curvature, or for short the  curved gravitational Vlasov-Poisson system,
\begin{equation}\label{vp_ext}
\begin{cases}
\partial_t f + \v \cdot \nabla_{\x} f + \left(\nabla_{\x}U - \sigma (\v \cdot \v) \x \right)\cdot \nabla_\v f=0,\cr
U(t,\x) = \dfrac{1}{4\pi} \displaystyle\iint_{\M^2} \rho(t,\y)\log\left(\dfrac{1+\sigma\x\cdot \y}{\sigma-\x\cdot \y}\right)d\y,\cr
\rho(t,\x) = \displaystyle\int_{T_{\x}(\M^2)} f(t,\x,\v)d\v.
\end{cases}
\end{equation}
In local coordinates, system \eqref{vp_ext} takes the form
\begin{equation}\label{vp_local}
\begin{cases}
\partial_t \tilde{f} + \omega_\alpha \partial_\alpha \tilde{f}+ \omega_\theta \partial_\theta \tilde{f}+ (\partial_\alpha \tilde{U} + \omega_{\theta}^2\ \! \sn\alpha\ \!\csn\alpha) \partial_{\omega_\alpha} \tilde{f}\cr
\hspace{4cm}+ \left(\dfrac{1}{\sn^2\alpha} \partial_\theta \tilde{U}- 2\omega_\alpha\omega_\theta\ctn\alpha\right) \partial_{ \omega_\theta} \tilde{f}=0\cr
\tilde{U}(t,\alpha,\theta)=\dfrac{1}{4\pi} \displaystyle\iint_{\M^2} \tilde{\rho}(t,\alpha',\theta')\log\left[\dfrac{1+\sigma\x(\alpha,\theta)\cdot\y(\alpha',\theta')}{\sigma-\x(\alpha,\theta)\cdot\y(\alpha',\theta')}\right]\!\sn\alpha'\ \! d\alpha'd\theta\cr
\tilde{\rho}(t,\alpha,\theta)= \displaystyle\int_{T_{\x}(\M^2)}\tilde{f}(t,\alpha,\theta)\ \!\sn\alpha\ \! d\omega_{\alpha}d\omega_{\theta}.\cr
\end{cases}
\end{equation}

\begin{rem}
In \cite{Marsden85} an approach in the spirit of geometric mechanics was used to define the Vlasov equation. In this sense, we can first define the Hamiltonian as 
$$
H\colon T^*\M^2\to\R,\ \ H(q,p)=\frac12|p|^2+U(q).
$$ 
Then, the conservation of a distribution function $f(t,q,p)$ along the phase-space trajectories of $H$, i.e. the law  $\frac d{dt}f=0$, can be written, using the canonical Poisson bracket $\{\cdot,\cdot\}$, as
$$
\partial_tf=\{f,H\},
$$
which gives the Vlasov-Poisson system if we require that the potential $U$ solves Poisson's equation,
$$
-\Delta_{\M^2}U(q)=\rho(q):=\int_{T_p\M^2}f(t,q,p)\;dp.
$$
\end{rem}
Although this form of the Vlasov-Poisson system in curved spaces is elegant and quite natural, its analysis seems to be difficult. For this reason, we will focus in the sequel on a special configuration of the mass distribution.

\section{Initial data along a geodesic}

In the rest of this paper we focus on the 1-dimensional case, and assume an initial distribution on $\M^2$ in which the particles lie initially on a geodesic, which is a great circle in $\S^2$ and a hyperbolic great circle in $\H^2$. In our model of hyperbolic geometry, a hyperbolic great circle is a hyperbola obtained by intersecting the upper sheet of the hyperboloid of two sheets with a plane through the origin of the extrinsic coordinate system. Therefore, we study a configuration of particles that obeys the following conditions:
\begin{enumerate}
\item[1.] the particles move on the geodesic $G$, where
$$
G:=
\begin{cases}
\left\{(x_1,x_2,x_3)\ \!|\, x_1^2 + x_2^2 = 1,\, x_3=0\right\},\quad\ \ \ \! \text{for $ \M^2 = \S^2 $}\cr
\left\{(x_1,x_2,x_3)\ \!|\, x_2^2 - x_3^2 = -1,\, x_1=0\right\},\quad \text{for $ \M^2 = \H^2 $},\cr
\end{cases}
$$
\item[2.] the velocity of each particle is always in the $\e_\theta$-direction for $ \M^2 = \S^2 $ and in the $ \e_\alpha $-direction for $ \M^2 = \H^2 $,
\end{enumerate}
where the geodesic $G$ has been chosen for convenience and without loss of generality. Notice that, since the gravitational force on each particle is directed along the geodesic, if the particles are initially aligned on $G$ with initial velocities along that geodesic, then they remain on $G$ for all time. 

For convenience, we consider
\[
	I=\left\{\begin{array}{ll}
		 [0,2\pi),&\text{if $ \M^2=\S^2 $}\\
		 (-\infty,+\infty),&\text{if $ \M^2=\H^2 $},
	\end{array}
	\right.
\]
and for any functions $ \phi(x,\omega),\psi(x,\omega) $ with $ x\in I,\omega\in\mathbb{R} $, we define
\[
	<\phi,\psi>\ =\int_\mathbb{R}\int_I\phi\psi\ \! d\alpha\ \! d\omega. 
\]

Let us further derive the equations when the motion of the particles is restricted to the geodesic $ G $. 
\subsection{Motion on a great circle of $ \S^2 $}
Let us notice that in the spherical case we have
	\[
		G = \left\{\x(\alpha,\theta): \alpha = \pi/2\right\}. 
	\]
This allows us to take the phase distribution after restriction in the form
\begin{equation}\label{geoS2}
\tilde{f}(t,\alpha,\theta,\wa,\wth) = \frac{\delta(\alpha-\pi/2)}{\sin^2(\alpha)}\otimes\delta(\omega_\alpha)g(t,\theta,\omega_\theta),
\end{equation}
where $ g $ is a distribution on $ G $. If we let $ \rg(t,\theta) = \int_\mathbb{R}g(t,\theta,\omega_\theta) d\omega_\theta$, then the following proposition, which is our first notable result, provides the equations of motion restricted to $G$.
\begin{prop}\label{1DS2}
If the phase distribution is taken as in \eqref{geoS2}, then $ g $ and $ \rg $ satisfy
\begin{equation}\label{reduced}
\left\{
	\begin{aligned}
		& \partial_t g + \wth\partial_\theta g+ F(t,\theta)\partial_\wth g = 0,\\
		& F(t,\theta) = W*_\theta\partial_\theta \rg,\\
		& \rg(t,\theta) = \int_\mathbb{R} g(t,\theta,\wth)d\wth
	\end{aligned}
\right.
\end{equation}
in the distributional sense, where $ W(\theta)=\frac{1}{2\pi}\log|\cot(\frac{\theta}{2})| $, $ \theta\in[0,2\pi) $, and $ *_\theta $ denotes convolution in the $ \theta $ variable.
\end{prop}
\begin{proof}
	Let $ \phi\in C_0^\infty([0,2\pi)\times\mathbb{R} )$ be a test function. Then there exists $ \Phi\in C_0^\infty([0,\pi]\times[0,2\pi)\times\mathbb{R}\times\mathbb{R}) $ such that $ \Phi|_{\alpha=\frac{\pi}{2},\wa=0}=\phi $. 
	Noticing that, after restriction, $ \wa \equiv 0 $ and $ \partial_\alpha \tilde{U}|_{\alpha=\frac{\pi}{2}}=0 $, substituting $ \tilde{f} $ into \eqref{vp_local} and testing with $ \Phi $, we have for the first term
	\[
		\int_\mathbb{R}\int_\mathbb{R}\int_0^{2\pi}\int_0^\pi \partial_t\tilde{f}\Phi d\alpha d\theta d\wa d\wth = \int_\mathbb{R}\int_0^{2\pi} \partial_tg\phi d\theta d\wth.
	\]
	The computations of the other terms are similar and in the end we obtain
	\begin{equation}
		<\partial_t g + \wth\partial_\theta g + \partial_\theta U_g\partial_\wth g,\phi>=0,
	\end{equation}
	where $ U_g=\tilde{U}|_{\alpha=\frac{\pi}{2}} $.
	
	We need to obtain now the explicit form of $ U_g $. Since for $ \x(\alpha,\theta),\y(\alpha',\theta')\in\S^2 $,
	\[
		\x\cdot\y = \sin\alpha\sin\alpha'\cos(\theta-\theta')+\cos\alpha\cos\alpha',
	\]
	we have that
	\begin{eqnarray*}
		&&\tilde{U}(t,\alpha,\theta)\\
		&=&\frac{1}{4\pi}\int_0^{2\pi}\int_0^\pi\rg(t,\theta') \delta(\alpha'-\pi/2)\log\frac{1+\sin\alpha\sin\alpha'\cos(\theta-\theta')+\cos\alpha\cos\alpha'}{1-\sin\alpha\sin\alpha'\cos(\theta-\theta')+\cos\alpha\cos\alpha'}d\alpha' d\theta'\\
		&=&\frac{1}{4\pi}\int_0^{2\pi}\rg(t,\theta')\log\frac{1+\sin\alpha\cos(\theta-\theta')}{1-\sin\alpha\cos(\theta-\theta')}d\theta'.
	\end{eqnarray*}
	Therefore,
	\begin{eqnarray*}
		&&U_g = \tilde{U}(t,\pi/2,\theta) \\
		&=&\frac{1}{4\pi}\int_0^{2\pi}\rg(t,\theta')\log\frac{1+\cos(\theta-\theta')}{1-\cos(\theta-\theta')}d\theta'.\\
		&=&\frac{1}{2\pi}\int_0^{2\pi}\rg(t,\theta')\log\left|\cot\frac{(\theta-\theta')}{2}\right|d\theta'.
	\end{eqnarray*}
	If we define the kernel
		\[
			W(\theta) = \frac{1}{2\pi}\log\left|\cot\frac{\theta}{2}\right|,
		\]
	we can write that
	\[
		\partial_\theta U_g = W*_\theta \partial_\theta\rg.
	\]
	Setting $ F(t,\theta)=\partial_\theta U_g $ completes the proof.
\end{proof}
\subsection{Restriction to a hyperbolic great circle of $ \H^2 $}
Let us note that, in the hyperbolic case,
\[
G = \left\{\x(\alpha,\theta): \theta = \pi/2,\ \alpha\in(-\infty,+\infty)\right\}.
\]
We can therefore take the phase distribution after restriction in the form
\begin{equation}\label{geoH2}
\tilde{f}(t,\alpha,\theta,\wa,\wth) = \frac{\delta(\theta-\pi/2)}{\sin^2(\alpha)}\otimes\delta(\wth)g(t,\alpha,\wa),
\end{equation}
where $ g $ is a distribution on $ G $. If we let $ \rg(t,\alpha) = \int_\mathbb{R}g(t,\alpha,\wa) d\wa$, then the proposition below provides the equations restricted to $G$.
\begin{prop}\label{1DH2}
	If the phase distribution is taken as in \eqref{geoH2}, then $ g $ and $ \rg $ satisfy
	\begin{equation*}
	\left\{
	\begin{aligned}
	& \partial_t g + \wa\partial_\alpha g+ F(t,\alpha)\partial_\wa g = 0,\\
	& F(t,\alpha) = W*_\alpha\partial_\alpha \rg,\\
	& \rg(t,\alpha) = \int_\mathbb{R} g(t,\alpha,\wa)d\wa
	\end{aligned}
	\right.
	\end{equation*}
	in the distributional sense, where $ W(\alpha)=\frac{1}{2\pi}\log|\coth(\frac{\alpha}{2})| $, $ \alpha\in\mathbb{R} $, and $ *_\alpha $ denotes convolution in the $ \alpha $ variable.
\end{prop}
\begin{proof}
	As in the spherical case, let $ \phi\in C_0^\infty(\mathbb{R}\times\mathbb{R} )$ be a test function. There exists $ \Phi\in C_0^\infty(\mathbb{R}\times[0,2\pi)\times\mathbb{R}\times\mathbb{R}) $ such that $ \Phi|_{\theta=\frac{\pi}{2},\wth=0}=\phi $. Noticing that after restriction, $ \wth\equiv 0 $ and $\partial_\theta\tilde{U}|_{\theta=\frac{\pi}{2}}=0 $, substituting $ \tilde{f} $ into \eqref{vp_local} and testing with $ \Phi $, we have
	\begin{equation}
		<\partial_t g + \wa\partial_\alpha g + \partial_\alpha U_g\partial_\wa g,\phi>=0,
	\end{equation}
	where $ U_g=\tilde{U}|_{\theta=\frac{\pi}{2}} $.
	
	Since for $ \x(\alpha,\theta),\y(\alpha',\theta')\in\H^2 $,
	\[
	\x\cdot\y = \sinh\alpha\sinh\alpha'\cos(\theta-\theta')-\cos\alpha\cos\alpha',
	\]
	we can write that
	\begin{eqnarray*}
	&&\tilde{U}(t,\alpha,\theta) \\
	&=& \frac{1}{4\pi}\int_\mathbb{R}\int_0^{2\pi} \rg(t,\alpha')\delta(\theta'-\pi/2)\log\left(\frac{1-\sinh\alpha\sinh\alpha'\cos(\theta-\theta')+\cosh\alpha\cosh\alpha'}{-1-\sinh\alpha\sinh\alpha'\cos(\theta-\theta')+\cosh\alpha\cosh\alpha'}\right)d\theta'd\alpha'\\
	&=& \frac{1}{4\pi}\int_\mathbb{R} \rg(t,\alpha')\log\left(\frac{1-\sinh\alpha\sinh\alpha'\cos(\theta-\pi/2)+\cosh\alpha\cosh\alpha'}{-1-\sinh\alpha\sinh\alpha'\cos(\theta-\pi/2)+\cosh\alpha\cosh\alpha'}\right)d\alpha'.
	\end{eqnarray*}
	Therefore,
	\begin{eqnarray*}
	&& U_g(t,\alpha) =\widetilde{U}(t,\alpha,\frac{\pi}{2})\\
	&=&  \frac{1}{4\pi}\int_\mathbb{R} \rg(t,\alpha')\log\left(\frac{1-\sinh\alpha\sinh\alpha'+\cosh\alpha\cosh\alpha'}{-1-\sinh\alpha\sinh\alpha'+\cosh\alpha\cosh\alpha'}\right)d\alpha' \\
	&=& \frac{1}{4\pi}\int_\mathbb{R} \rg(t,\alpha')\log\left(\frac{1+\cosh(\alpha-\alpha')}{-1+\cosh(\alpha-\alpha')}\right)d\alpha' \\
	&=& \frac{1}{4\pi}\int_\mathbb{R} \rg(t,\alpha')\log\left(\coth^2\left(\frac{\alpha-\alpha'}{2}\right)\right)d\alpha' \\
	&=& \frac{1}{2\pi}\int_\mathbb{R} \rg(t,\alpha')\log\left|\coth\left(\frac{\alpha-\alpha'}{2}\right)\right|d\alpha'.
	\end{eqnarray*}
	Setting the kernel
	\[
		W(\alpha)=\frac{1}{2\pi}\log\left|\coth\frac{\alpha}{2}\right|,
	\]
	and the force
	\[
		F(t,\alpha) = \partial_\alpha U_g,
	\]
	completes the proof.
\end{proof}
We can combine the results of Proposition \ref{1DS2} and Proposition \ref{1DH2} by replacing both $ (\theta,\wth) $ in Proposition \ref{1DS2} and $ (\alpha,\wa) $ in Proposition \ref{1DH2} by $ (x, v)\in I\times \mathbb{R} $. Doing so produces the following proposition.
\begin{prop}\label{eq_geo}
	Let $ f $ be a phase space distribution along any geodesic $ G $ of $ \M^2 $. Then the system restricted to $ G $ becomes
\begin{equation}
\label{pde_conv}
\left\{
		\begin{aligned}
		& \partial_t f + v\partial_x f+ F(t,x)\partial_v f = 0,\\
		& F(t,x) = W*\partial_x \rho,\\
		& \rho(t,x) = \int_\mathbb{R} f(t,x,v)dv,
		\end{aligned}
		\right.
\end{equation}
	where $  (x, v) $ with $ (x,v)\in I\times\mathbb{R}, W(x) = \frac{1}{2\pi}\log|\ctn(\frac{x}{2})| $, and $ * $ denotes the convolution in $ x $.
\end{prop}
Several quantities are conserved along the solutions of the above equations, and the next result summarizes them.

\begin{prop}
Let 
$f(t,x,v)$
be a solution of \eqref{pde_conv} such that $f$ has compact support in $v$, has compact support in $x$ for $\M^2 = \H^2$, and is $2\pi$-periodic in $x$ for $\M^2 = \S^2$. Then system \eqref{pde_conv} can be alternatively written as
\begin{equation}
\label{pde}
\partial_t f + v\partial_x f+ \partial_x U\partial_v f = 0,
\end{equation}
where $U(t,x) = \dfrac{1}{2\pi}\displaystyle\int_I \rho(t,y)\log\left|\ctn\left(\frac{x-y}{2}\right)\right|dy$ and $\rho(t,x) = \displaystyle\int_{\R}f(t,x,v)dv$,
and the following quantities are conserved:

\smallskip

(i) the total number of particles,
$$N:=
\int_{\R}\int_{I} f(t,x,v)\,dx dv,$$ 

(ii) the total mechanical energy, 
\begin{equation*}
E:=
\dfrac{1}{2}\int_{\R}\int_{I} f(t, x, v) v^2\,dx dv
 - \int_{I} U(t,x) \rho(t,x)\, dx,
\end{equation*}

(iii) the entropy,
$$
S := -\int_{\R}\int_{I} f(t, x, v)\log\left[f(t,x,v)\right]\, dx dv,
$$

(iv) the Casimirs, 
$$
\mathcal{C} := 
\int_{\R}\int_{I} A(f(t,x,v))\,dx dv,
$$
where $A$ is an arbitrary smooth function. 
\end{prop}
\begin{proof}
To show that the total number $N$ of particles is conserved, we integrate \eqref{pde} over the phase space as follows,
\begin{equation*}
\begin{array}{lll}
0&=& \displaystyle\int_{\R}\int_{I} [\partial_t f + v \partial_x f+ \partial_x U\partial_{v} f]\ \! dx\ \! dv\\[5mm]
&=& \displaystyle\int_{\R}\int_{I}\partial_t f\ \! dx\ \! dv + \int_{\R}\int_{I} v \partial_x f\ \! dx\ \! dv+ \int_{\R}\int_{I} \partial_x U\partial_{v} f\ \! dx\ \! dv\\[5mm]
&=& \dfrac{d}{dt}\displaystyle\int_{\R}\displaystyle\int_{I} f\ \! dx\ \! dv + \displaystyle\int_{\R}\displaystyle\int_{I}v \partial_x f\ \! dx\ \! dv+ \displaystyle\int_{I}\displaystyle\int_{\R} \partial_x U\partial_{v} f\ \!dvdx\\[5mm]
&=& \dfrac{d}{dt}\displaystyle\int_{\R}\displaystyle\int_{I} f\ \! dx\ \! dv - \displaystyle\int_{\R}\displaystyle\int_{I} f\ \! \partial_x v\ \! dx\ \! dv- \displaystyle\int_{I}\displaystyle\int_{\R} f \ \!\partial_{v_x}(\partial_x U)\ \!dvdx\\[5mm]
&=&\dfrac{d}{dt}\displaystyle\int_{\R}\displaystyle\int_{I} f\ \! dx\ \! dv,
\end{array}
\end{equation*}
where we used that $f$ has compact support in $v$, that $f(0)=f(2\pi)$ in the positive curvature case, and that $f$ has compact support in $x$ in the negative curvature case.

For the conservation of total mechanical energy, 
$$E=\frac{1}{2}\displaystyle\int_{\R}\displaystyle\int_{I} fv^2dx dv-\displaystyle\int_{I} U\rho dx,$$ we multiply \eqref{pde} by $\frac{1}{2}v^2$ and integrate over phase space.

Since $f$ satisfies relation \eqref{pde}, by the chain rule so does the Casimir $A(f)$. Therefore, we can integrate 
$$
\partial_t A(f) + v\partial_x A(f) + \partial_x U \partial_{v} A(f)=0
$$
over the phase space to yield our conservation law. Entropy is a Casimir, so its conservation follows directly from the conservation of Casimirs. 
\end{proof}

We are now also in a position to prove the following result.

\begin{prop}\label{equil}
Any distribution of the form $f(t,x,v) = f^0(v)$ is a spatially homogeneous equilibrium solution of equation \eqref{pde_conv}.
\end{prop}
\begin{proof}
By definition, any stationary solution must satisfy \eqref{pde} with $\partial_t f=0$, i.e.
\begin{equation}
\label{stationary}
v \partial_x f + F\ \!\partial_{v} f=0.
\end{equation}
Consider a spatially homogeneous distribution function $f=f^0(v)$. For this form of $f$, we get
\begin{equation}
\partial_x f^0=0,
\end{equation}
so the first term in \eqref{stationary} is 0. We can use the expression of $\rho$ in \eqref{pde_conv} to calculate
\begin{equation}
\rho(t,x) = \displaystyle\int f^0 \;dv = \rho^0 \ {\rm (constant)}.
\end{equation}
Consequently, we obtain that the force due to the homogeneous distribution is
\begin{equation}
F=W*\partial_x\rho^0=0,
\end{equation}
since $\rho^0$ is a constant. Therefore the second term in \eqref{stationary} vanishes and we conclude that $f(t,x, v) = f^0(v)$ is a spatially homogeneous equilibrium (stationary) solution to equation \eqref{pde}.
\end{proof}

\section{Local well-posedness}
The purpose of this section is to construct a solution to \eqref{pde_conv} in the class of analytic functions such that to apply to it the Penrose condition for linear stability. In a forthcoming paper, we will study the existence of weak solutions and  strong solutions local in time in the class of Sobolev spaces. This is in contrast with the Vlasov equation given by the Dirac  potential studied in \cite{analyticsol}, where the singularity is stronger.  Since we are working with system \eqref{pde_conv}, we rewrite it here in a slightly different but equivalent form,
\begin{equation}\label{vp}
\left\{\begin{aligned}
&\partial_t f + v\partial_x f + (W* \partial_x\rho) \partial_v f = 0 \\
&\rho = \int_\mathbb{R} f dv. \\
&W(x) = \frac{1}{2\pi} \log\left|\ctn\left(\frac{x}{2}\right)\right|.
\end{aligned}
\right.
\end{equation}
We will construct our solution in the class of Gevrey functions $ G^s$. For $s\ge 1 $, set
\[
	G^s(\textbf{x}):=\left\{f(\textbf{x})\ \! |\ \! \forall \textbf{k}=(k_1,k_2)\in\mathbb{N}^2, \exists M_{\textbf{k}}>0 \ s.t. \ |D^{\textbf{k}} g|\le M_{\textbf{k}}(k_1!k_2!)^s\right\}.
\]
Note that $G^1$ is the set of analytic functions and $G^1\subset G^s$. Let $ P $ be the cone of polynomial functions with positive coefficients. Fixing $ s\ge1 $, we define the operator $ D^{a,s} $ on $ P $ as
\[
D^{a,s} (\lambda^n) = \lambda^{n-a} \left(\frac{n!}{(n-a)!}\right)^s,
\]
where $a, n\in\mathbb{N} $.
\begin{rem}
When $ s=1 $ the operator $ D^{a,s} $ reduces to the classic $a$-derivatives operator w.r.t $ \lambda $, which will help us define a suitable function space to find the solution.
\end{rem}
For any function $ f(t,x,v) $ and positive real numbers $ (\lambda_0,K) $, define
\[
f_{k,l} = \partial_x^k\partial_v^lf,\quad |f(t)|_{\lambda} := \sum_{k+l\ge0 } \frac{\lambda^{k+l}}{(k!l!)^s}|f_{k,l}|_{L^\infty_{x,v}}, \quad |f(t)|_{\lambda,a} :=D^{a,s}|f(t)|_{\lambda},
\]
\[
\lambda(t)=\lambda_0-(1+K)t,\ H_{f}(t) := \sum_{a\ge0} \frac{|f|_{\lambda(t),a}(t)}{(a!)^{2s}}, \  \tilde{H}_{f}(t) := \sum_{a\ge0} a^{2s}(a+1)^{s-1}\frac{|f|_{\lambda(t),a}(t)}{(a!)^{2s}}.
\]
Note that all the above norms still make sense when the functions depend only on the variable $ x  $ or $ v $ by setting $ k =0$ or $ l=0 $, respectively.

Let us denote $ C_w=\|W\|_{L^1} $ and assume that  the initial condition $ f_i(x,v) $ of \eqref{vp} satisfies
\begin{equation}\label{IC}
	h_0 := H_{f_i}(0) < \infty.
\end{equation}
For all $ M>h_0>0 $, we set $ \tilde{M}=\frac{1}{18^sC_w}\ln\frac{M}{h_0} $, and for all $ T>0 $, we define the space
\[
	\mathcal{H}_{T,M} = \{f(t,x,v): \sup_{t\in[0,T]} H_{f}(t)\le M, \int_0^T\tilde{H}_f(t)dt \le \tilde{M} \},
\]
which we endow with the norm
\[
	\|f\|_T=\sup_{t\in[0,T]} H_{f}(t)+\int_0^T\tilde{H}_f(t)dt.
\]

We are now in a position to state our second notable result.
\begin{thm}\label{mainthem}
	Let $s\geq1$ and let $ f_i \in G^s$ be the initial condition of \eqref{vp} satisfying $ h_0< \infty $. Then there exist a time $ T>0 $ and a constant $ M>h_0 $ such that \eqref{vp} has a solution $ f\in\mathcal{H}_{T,M} \subset G^s$.
\end{thm}

The proof of this theorem is tedious and goes through several steps. First, we need the following propositions and lemmas.

\subsection{Preliminary estimates}
In this subsection we will set the background for the proof of Theorem \ref{mainthem}. Our first result generalizes the Leibniz rule and provides some useful inequalities for the operator $D^{a,s}$.

\begin{prop}
	The operator $ D^{a,s} $ is linear and for any $ a,b\in\mathbb{N} $, $ p(\lambda),q(\lambda)\in P, $ 
	\begin{equation}\label{prop1}
	D^{a,s}(D^{b,s} p(\lambda)) = D^{a+b,s} p(\lambda),
	\end{equation}
	\begin{equation}\label{prop2}
	(1+a)^{1-s}D^{a,s}(p(\lambda)q(\lambda) )\le \sum_{k=0}^{a} (C_a^k)^sD^{k,s}p(\lambda)D_\lambda^{a-k}q(\lambda) \le D^{a,s}(p(\lambda)q(\lambda) ).
	\end{equation}
\end{prop}
\begin{proof}
	The proof of \eqref{prop1} is straightforward. For \eqref{prop2}, it is sufficient to prove that the inequalities hold for $ p(\lambda)=\lambda^m, q(\lambda)=\lambda^n $. For this, notice that
	\begin{eqnarray*}
		&&\sum_{k=0}^{a} (C_a^k)^sD^{k,s}p(\lambda)D^{a-k,s}q(\lambda) 
		=\sum_{k=0}^a(a!C_m^kC_n^{a-k})^s\lambda^{m+n-a}\\
		&\le&\left(\sum_{k=0}^{a}a!C_m^kC_n^{a-k}\right)^s\lambda^{m+n-a}
		= \left(\frac{(m+n)!}{(m+n-a)!}\right)^s\lambda^{m+n-a} 
		= D^{a,s}(p(\lambda)q(\lambda)).
	\end{eqnarray*}
	Meanwhile, since
	\[
		\left(\frac{1}{a+1}\sum_{k=0}^{a}a!C_m^kC_n^{a-k}\right)^s\le\frac{1}{a+1}\sum_{k=0}^{a}(a!C_m^kC_n^{a-k})^s,
	\]
	we have
	\begin{eqnarray*}
		&&\sum_{k=0}^{a} (C_a^k)^sD^{k,s}p(\lambda)D^{a-k,s}q(\lambda)
		=\sum_{k=0}^a(a!C_m^kC_n^{a-k})^s\lambda^{m+n-a}\\
		&\ge&(a+1)^{1-s}\left(\sum_{k=0}^{a}a!C_m^kC_n^{a-k}\right)^s\lambda^{m+n-a}
		=(a+1)^{1-s}D^{a,s}(p(\lambda)q(\lambda)).
	\end{eqnarray*}
This remark completes the proof. \end{proof} 
It is important to note that estimates on the distribution function $f(t,\alpha,\omega)$ in $H_f$ or $\tilde{H}_f$ norms automatically yield the same estimates for the corresponding density function $\rho$. Indeed, since  $ \alpha(v)\ge 0$ and $ \int_\mathbb{R}\alpha(v)dv\le 1 $,  we have
	\[
	|\partial_x^k\rho|_{L^\infty_x}=|\int_{\mathbb{R}}\alpha\partial_x^k fdv|_{L^\infty_x}\le|\int_{\mathbb{R}}\alpha(v) dv||\partial_x^k f|_{L^\infty_{x,v}}=|\partial_x^k f|_{L^\infty_{x,v}}.
	\]
	The following result enables us to estimate the distribution function when multiplied by the weight function $\gamma$.
\begin{prop}\label{HgammaG}
Assume $ \gamma(v) $ satisfies
	\begin{equation}\label{Cgamma}
	C_\gamma := \sum_{k\ge0} \sum_{0\le l\le k} C_k^l\lambda_0^l\sup_{0\le n\le k}\left(|\gamma_n|_{L^\infty_v}\right)<\infty,
	\end{equation}
	then for any $ f(t,x,v) $ we have that
	\[
	H_{\gamma f} \le C_\gamma H_{f} \ \ {\rm and}\ \ \tilde{H}_{\gamma f}\le C_\gamma \tilde{H}_{f}.
	\]
\end{prop}
\begin{proof} First, using the definition of the norms and the Leibniz rule, we have
\begin{eqnarray*}
	H_{\gamma f} &=& \sum_{a\ge0}\sum_{k+l\ge a} \frac{1}{(a!)^{2s}}\frac{\lambda^{k+l-a}}{(k!l!)^s}\left(\frac{(k+l)!}{(k+l-a)!}\right)^s|(\gamma f)_{k,l}|_{L^\infty_{x,v}} \\
	\label{2}&=& \sum_{a\ge0}\sum_{k+l\ge a} \frac{1}{(a!)^{2s}}\frac{\lambda^{k+l-a}}{(k!l!)^s}\left(\frac{(k+l)!}{(k+l-a)!}\right)^s|\sum_{b=0}^{l}C_l^{l-b}f_{k,l-b}\gamma_{b}|_{L^\infty_{x,v}} \\
	&\le&\sum_{a\ge0}\sum_{k+l\ge a} \frac{1}{(a!)^{2s}}\frac{\lambda^{k+l-a}}{(k!l!)^s}\left(\frac{(k+l)!}{(k+l-a)!}\right)^s \sum_{0\le b\le l}C_l^b|f_{k,b}|_{L^\infty_{x,v}} \\&&\sup_{0\le n\le l}\left(|\gamma_n|_{L^\infty_v}\right)\\
	&=&\sum_{0\le b\le l}\sum_{a\ge0}\sum_{k+l\ge a} \frac{1}{(a!)^{2s}}\frac{\lambda^{k+b-a}}{(k!b!)^s}\left(\frac{(k+b)!}{(k+b-a)!}\right)^s|f_{k,b}|_{L^\infty_{x,v}}\\
	&&C_l^b\left(\frac{b!}{l!}\frac{(k+l)!}{(k+b)!}\frac{(k+b-a)!}{(k+l-a)!}\right)^s\lambda^{l-b}\sup_{0\le n\le l}\left(|\gamma_n|_{L^\infty_v}\right)\\
	&\le& \sum_{0\le b\le l}\sum_{a\ge0}\sum_{k+l\ge a} \frac{1}{(a!)^{2s}}\frac{\lambda^{k+b-a}}{(k!b!)^s}\left(\frac{(k+b)!}{(k+b-a)!}\right)^s|f_{k,b}|_{L^\infty_{x,v}}\\
	&& C_l^b\lambda^{l-b}\sup_{0\le n\le l}\left(|\gamma_n|_{L^\infty_v}\right) 
\end{eqnarray*}
Second, since $ b\le l $, for each fixed $ l $  we have
\begin{eqnarray*}
&&\sum_{a\ge0}\sum_{k+l\ge a} \frac{1}{(a!)^{2s}}\frac{\lambda^{k+b-a}}{(k!b!)^s}\left(\frac{(k+b)!}{(k+b-a)!}\right)^s|f_{k,b}|_{L^\infty_{x,v}} \\
&\le&\sum_{a\ge0}\sum_{k+b\ge a} \frac{1}{(a!)^{2s}}\frac{\lambda^{k+b-a}}{(k!b!)^s}\left(\frac{(k+b)!}{(k+b-a)!}\right)^s|f_{k,b}|_{L^\infty_{x,v}} \le H_f.
\end{eqnarray*}
Therefore
\begin{eqnarray*}
H_{\gamma f}&\le& \sum_{l\ge0} \sum_{0\le b\le l}H_f\cdot C_l^b \lambda^b \sup_{0\le n\le l}\left(|\gamma_n|_{L^\infty_v}\right) \le C_\gamma H_f.
\end{eqnarray*}
The proof for $ \tilde{H}_{\gamma f} $ is similar.
\end{proof}
The next result is the first step in the key estimate which shows that the contraction map is well-defined.
\begin{lem}\label{lemma1}
	Given $ \rho(t,x) $ and $ \gamma(v) $ satisfying \eqref{Cgamma}, let $ g $ be the solution to 
	\begin{equation}\label{iteration}
	\partial_t g + v\partial_x g + (W* \partial_x\rho) (\partial_v g+\gamma g) = 0,
	\end{equation}
	with $ g(0,x,v)=f_i(x,v) $.
	 Then we have that
	\[
	\partial_t|g|_{\lambda,a} \le a^s|g|_{\lambda,a}+\lambda|g|_{\lambda,a+1}+C_wD^{a,s}(|\rho|_{\lambda,1}|g|_{\lambda,1}+|\rho|_{\lambda,1}|\gamma g|_\lambda).
	\]
\end{lem}
\begin{proof}
	Differentiate \eqref{iteration} $ k $ times w.r.t. $ x $ and $ l $ times w.r.t $ v $ and notice that
	\begin{eqnarray*}
		&&\partial_t(g_{k,l})+v\partial_x(g_{k,l})+(W*\partial_x\rho)\partial_v(g_{k,l}) \\
		&=&-lg_{k+1,l-1}+\sum_{m=0}^{k-1}C_k^m(W*\partial_x^{k-m+1}\rho)(g_{m,l+1})+\sum_{m=0}^{k}C_k^m(W*\partial_x^{k-m+1}\rho)(\gamma g)_{m,l}.
	\end{eqnarray*}
	Multiplying by $ D^{a,s} (\frac{\lambda^{k+l}}{(k!l!)^s})=\frac{\lambda^{k+l-a}}{(k!l!)^s}\left(\frac{(k+l)!}{(k+l-a)!}\right)^s$, summing over $ k+l\ge a $, and using the method of characteristics, we obtain the inequality
	\begin{equation}
	\frac{d}{dt}|g|_{\lambda,a} \le \sum_{k\ge a-l} \frac{\lambda^{k+l-a}}{(k!(l-1)!)^s}\left(\frac{(k+l)!}{(k+l-a)!}\right)^s|g_{k+1,l-1}|_{L^\infty_{x,v}} \label{est1}
	\end{equation}
	\begin{equation}
	+ C_w\sum_{k\ge a-l} \sum_{m\le k-1} D^{a,s} (\lambda^{k+l}) \frac{1}{(l!m!(k-m)!)^s}|g_{m,l+1}|_{L^\infty_{x,v}}|\partial_x^{k-m+1}\rho|_{L^\infty_x} \label{est2} 
	\end{equation}
\begin{equation}
+ C_w\sum_{k\ge a-l} \sum_{m\le k} D^{a,s} (\lambda^{k+l}) \frac{1}{(l!m!(k-m)!)^s}|(\gamma g)_{m,l}|_{L^\infty_{x,v}}|\partial_x^{k-m+1}\rho|_{L^\infty_x}. \label{est3}
	\end{equation}
	The first term on the right hand side can be controlled by changing $ k+1 $ to $ k $ and $ l-1 $ to $ l $, as follows,
		\begin{eqnarray*}
		\eqref{est1} &=& \sum_{k\ge a-l} k^s \frac{\lambda^{k+l-a}((k+l)!)^s}{(k!l!(k+l-a)!)^s}|g_{k,l}|_{L^\infty_{x,v}}\\
		&=& \sum_{k\ge a-l} 2^{s}((1/2)a+(1/2)(k-a))^s \frac{\lambda^{k+l-a}((k+l)!)^s}{(k!l!(k+l-a)!)^s}|g_{k,l}|_{L^\infty_{x,v}}\\
		&\le& 2^{s-1}a^s\sum_{k\ge a-l}\frac{\lambda^{k+l-a}((k+l)!)^s}{(k!l!(k+l-a)!)^s}|g_{k,l}|_{L^\infty_{x,v}} + 2^{s-1}\sum_{k\ge a-l+1} \frac{\lambda^{k+l-a}((k+l)!)^s}{(k!l!(k+l-a-1)!)^s}|g_{k,l}|_{L^\infty_x} \\
		&\le&  a^s|g|_{\lambda,a}+ \lambda|g|_{\lambda,a+1}.
	\end{eqnarray*}
	The second term is $ \eqref{est2} = C_wD^{a,s} Y $, where 
	\[
	Y=\sum_{k\ge a-l} \sum_{m\le k-1}\lambda^{k+l}\frac{1}{(l!m!(k-m)!)^s}|g_{m,l+1}|_{L^\infty_{x,v}}|\partial_x^{k-m+1}\rho|_{L^\infty_x}.
	\]
	Because $Y$ is a polynomial in $ \lambda $ with positive coefficients, we can get an estimate on $ D^a_\lambda Y$ just by estimating $ Y $ itself,
	\begin{eqnarray*}
		Y&=&\sum_{k\ge a-l} \sum_{m\le k-1}\lambda^{k+l}\frac{1}{(l!m!(k-m)!)^s}|g_{m,l+1}|_{L^\infty_{x,v}}|\partial_x^{k-m+1}\rho|_{L^\infty_x} \\
		&=& \sum_{m+l\ge a-k}\sum_{k\ge2}\frac{\lambda^{k+l+m-2}}{((l-1)!m!(k-1)!)^s} |g_{m,l}|_{L^\infty_{x,v}}|\partial_x^k\rho|_{L^\infty_x} \\
		&\le& \sum_{m+l\ge a-k}\sum_{k\ge2} \frac{\lambda^{l+m-1}(l+m)^s}{(l!m!)^s}|g_{m,l}|_{L^\infty_{x,v}} \frac{\lambda^{k-1}}{((k-1)!)^s}|\partial_x^k\rho|_{L^\infty_x} \\
		&\le& \sum_{m+l\ge 1}  \frac{\lambda^{l+m-1}(l+m)^s}{(l!m!)^s}|g_{m,l}|_{L^\infty_{x,v}} \sum_{k\ge1}  \frac{\lambda^{k-1}}{((k-1)!)^s}|\partial_x^k\rho|_{L^\infty_x}
		= |\rho|_{\lambda,1}|g|_{\lambda,1}.
	\end{eqnarray*}
	Hence
	\[
	\eqref{est2} = C_wD^{a,s} Y \le C_wD^{a,s}(|\rho|_{\lambda,1}|g|_{\lambda,1}).
	\]
	Similarly, 
	\[
	\eqref{est3} \le C_wD^{a,s}\tilde{Y},
	\]
	where
	\begin{eqnarray*}
		\tilde{Y}&=& \sum_{k\ge a-l} \sum_{m\le k} \lambda^{k+l} \frac{1}{(l!m!(k-m)!)^s}|(\lambda g)_{m,l}|_{L^\infty_{x,v}}|\partial_x^{k-m+1}\sigma|_{L^\infty_x} \\
		&=& \sum_{k\ge a-l} \sum_{m\le k} \frac{\lambda^{m+l}}{(m!l!)^s}|(\lambda g)_{m,l}|_{L^\infty_{x,v}}\frac{\lambda^{k-m}}{((k-m)!)^s}|\partial_x^{k-m+1}\rho|_{L^\infty_x} 
		\le |\rho|_{\lambda,1}|\gamma g|_\lambda.
	\end{eqnarray*}
	Consequently we have
	\[
	\eqref{est3} \le  C_wD^{a,s}(|\rho|_{\lambda,1}|\gamma g|_{\lambda}).
	\] 
	The proof follows by combining estimates $ \eqref{est1}, \eqref{est2}$, and  $\eqref{est3}$.
\end{proof}
\begin{lem}\label{lemma2}
	Assume $ M>0 $, $ \gamma(v) $ satisfying \eqref{Cgamma} and $ \tilde{M} = \frac{1}{18^sC_w+2C_\gamma}\ln\frac{M}{h_0} $ are given.
	Let $ \lambda(t)=\lambda_0-(1+K)t $, $ 0\le t\le T $.  Then there exists $ K $ such that for any $ \rho$ with
	\[
	\sup_{0\le t\le T}H_\rho(t) \le M \ \ {\rm and}\ \ \int_t^T \tilde{H}_\rho(t)dt \le \tilde{M}, 
	\]
	the solution to \eqref{iteration} is such that
	$
	g\in\mathcal{H}_{T,M}
	$
	and 
	\[
	\partial_tH_g\le(\lambda_0-K+(4\cdot9^s+C_\gamma)C_wH_\rho)\tilde{H}_g+(18^s+2C_\gamma)C_wH_g\tilde{H}_\rho.
	\]
\end{lem}
\begin{proof}
	By Lemma \ref{lemma1},
	\begin{eqnarray*}
		\partial_t |g|_{\lambda(t),a}&\le& a^s|g|_{g_{\lambda(t),a}}+\lambda(t)|g|_{\lambda(t),a+1}+C_wD^{a,s}(|\rho|_{\lambda(t),1}|g|_{\lambda(t),1}+|\rho|_{\lambda(t),1}|\gamma g|_{\lambda(t)}) \\
		&&-(1+K)|g|_{\lambda(t),a+1} \\
		&=& a^s|g|_{g_{\lambda(t),a}}+(\lambda_0-1-K)|g|_{\lambda(t),a+1}+C_wD^{a,s}(|\rho|_{\lambda(t),1}|g|_{\lambda(t),1}+|\rho|_{\lambda,1}|\gamma g|_{\lambda}).
	\end{eqnarray*}
	Multiplying the above inequality by $ \frac{1}{(a!)^{2s}} $ and summing w.r.t $ a $ yields
	\begin{eqnarray}
	\partial_tH_{g} &\le& \sum_{a\ge 1} \frac{a^s}{(a!)^{2s}} |g|_{\lambda(t),a} + (\lambda_0-1-K)\sum_{a\ge 1} \frac{|g|_{\lambda(t),a+1}}{(a!)^{2s}} \label{kes1}\\
	&&+C_w\sum_{a\ge 0} \sum_{0\le k \le a} \frac{(a+1)^{s-1}(C_a^k)^s}{(a!)^{2s}} |g|_{\lambda(t), k+1} |\rho|_{\lambda(t), a-k+1} \label{kes2}\\
	&&+C_w\sum_{a\ge 0} \sum_{0\le k \le a} \frac{(a+1)^{s-1}(C_a^k)^s}{(a!)^{2s}} |\gamma g|_{\lambda(t), k} |\rho|_{\lambda(t), a-k+1}. \label{kes3}
	\end{eqnarray}
	It is easy to see that 
	\[
	\eqref{kes1} \le \tilde{H}_g+(\lambda_0-1-K)\tilde{H}_g \le (\lambda_0-K)\tilde{H}_g.
	\]
	Then changing the index $ k+1 $ to $ k $, we obtain
	\begin{eqnarray*}
		\eqref{kes2} &=&C_w\sum_{a\ge 0} \sum_{1\le k \le a+1} \frac{(a+1)^{s-1}(C_a^{k-1})^s}{(a!)^{2s}} |g|_{\lambda(t), k} |\rho|_{\lambda(t), a-k+2} \nonumber\\
		&=& C_w\sum_{a\ge 0} \sum_{1\le k \le a-2} \frac{(a+1)^{s-1}(C_a^{k-1})^s}{(a!)^{2s}} |g|_{\lambda(t), k} |\rho|_{\lambda(t), a-k+2} \\
		&&+C_w|\rho|_{\lambda(t),1} \sum_{a\ge 0} \frac{(a+1)^{s-1}|g|_{\lambda(t),a+1}}{(a!)^{2s}}+C_w|\rho|_{\lambda(t),2} \sum_{a\ge 0} \frac{(a+1)^{s-1}a^{2s}}{(a!)^{2s}}|g|_{\lambda(t),a} \nonumber\\
		&&+C_w|\rho|_{\lambda(t),3} \sum_{a\ge 1} \frac{(a+1)^{s-1}(a(a-1)/2)^s}{(a!)^{2s}}|g|_{\lambda,a-1}. 
	\end{eqnarray*}
Finally, changing the index $ a-k+2 $ to $ a $, we obtain
	\begin{eqnarray}
	\eqref{kes2}&=& \nonumber\\
	 &&\label{I1}C_w\sum_{a\ge 4} \sum_{k\ge1} \frac{(a+k-1)^{s-1}}{((a+k-2)!)^s}\frac{1}{((k-1)!(a-1)!)^s}|g|_{\lambda,k}|\rho|_{\lambda,a} \\
	&+& C_w{\rho}_{\lambda,1}\sum_{a\ge0}\frac{(a+1)^{s-1}}{(a!)^{2s}}|g|_{\lambda,a+1}+C_w|\rho|_{\lambda,2}\sum_{a\ge0}\frac{(a+2)^{s-1}a^s}{(a!)^{2s}}|g|_{\lambda,a} \nonumber\\
	\label{I2}&+& C_w|\rho|_{\lambda,3}\sum_{a\ge2} \frac{a^{s-1}(a-1)^{2s}}{((a-1)!)^{2s}}|g|_{\lambda,a-1} \frac{a(a+1)^{s}}{2^s(a+1)(a-1)^sa^{2s}}
	\end{eqnarray}
	and
	\begin{eqnarray*}
		\eqref{I1} &=& C_w\sum_{k\ge 1}\sum_{a\ge 4} \frac{|g|_{\lambda,k}}{(k!)^{2s}} a^{2s}\frac{|\rho|_{\lambda,a}}{(a!)^{2s}} \left(\frac{kk!(a-1)!}{(a+k-2)!}\right)^s(a+k-1)^{s-1} \\
		&\le& C_w\sum_{k\ge 1}\sum_{a\ge 4} \frac{|g|_{\lambda,k}}{(k!)^{2s}} (a+1)^{s-1}a^{2s}\frac{|\rho|_{\lambda,a}}{(a!)^{2s}} \left(\frac{kk!(a-1)!}{(a+k-2)!}\right)^s(k+1)^{s-1}\\
		&\le&C_w\sum_{k\ge 1}\sum_{a\ge 4} \frac{|g|_{\lambda,k}}{(k!)^{2s}} (a+1)^{s-1}a^{2s}\frac{|\rho|_{\lambda,a}}{(a!)^{2s}} \left(\frac{k(k+1)!(a-1)!}{(a+k-2)!}\right)^s.
	\end{eqnarray*}
	For $ k\ge3, a\ge 4 $, we have that
	\[
	\frac{k(k+1)!(a-1)!}{(a+k-2)!} = \frac{(a-1)!}{(a-1)!}\cdot1\cdot2\cdot3\cdot \frac{4\cdots(k+1)k}{a\cdot(a+1)\cdots(a+k-2)} \le 6.
	\]
	Similarly, for $ k=2, a\ge4 $, we obtain that $ \frac{k(k+1)!(a-1)!}{(a+k-2)!}\le 18 $, while for $ k=1, a\ge4 $, we are led to the inequality $ \frac{k(k+1)!(a-1)!}{(a+k-2)!}\le 2 $. Then
	\[
	\eqref{I1} \le 18^sC_wH_g\tilde{H}_{\rho}.
	\]
	Observing that for all $s\geq1$ and $ a \ge 2$ we have 
	\begin{equation}\label{condition on s}
	\frac{a(a+1)^{s}}{2^s(a+1)(a-1)^sa^{2s}}\le \frac{(a+1)^{s+1}}{2^s(a+1)(a-1)^sa^{2s}} \le \frac{1}{2^s}\left(\frac{a+1}{a^2}\right)^s\le1,
	\end{equation}
	we can conclude that
	\[
	\eqref{kes2} \le 18^sC_wH_g\tilde{H}_{\rho}+ 4C_w9^sH_\rho \tilde{H}_g.
	\]
	Notice further that 
	\begin{eqnarray*}
		\eqref{kes3} &=& C_w\sum_{a\ge0} \sum_{1\le k \le a-1} \frac{(a+1)^{s-1}(C_a^k)^s}{(a!)^{2s}} |\gamma g|_{\lambda(t), k} |\rho|_{\lambda(t), a-k+1} \\
		&& + C_w|\gamma g|_{\lambda}\sum_{a\ge 0} \frac{(a+1)^{s-1}}{(a!)^{2s}}|\rho|_{\lambda,a+1}+C_w|\rho|_{\lambda,1}\sum_{a\ge 0} \frac{(a+1)^{s-1}}{(a!)^{2s}}|\gamma g|_{\lambda,a} \\
		&\le & C_w\sum_{a\ge0} \sum_{1\le k \le a-1} \frac{(a+1)^{s-1}(C_a^k)^s}{(a!)^{2s}} |\gamma g|_{\lambda(t), k} |\rho|_{\lambda(t), a-k+1} \\
		&& + C_w(H_{\lambda g}\tilde{H}_\rho + \tilde{H}_{\lambda g}H_\rho).
	\end{eqnarray*}
	Changing $ a-k+1 $ to $ k $, we are led to
	\begin{eqnarray}
	\eqref{kes3} &\le& C_w\label{subkes1}\sum_{a\ge 2}\sum_{k\ge 1}\frac{(a+k)^{s-1}}{((a+k-1)!k!(a-1)!)^s} |\gamma g|_{\lambda,k} |\rho|_{\lambda,a} \\
	&&+C_w(H_{\lambda g}\tilde{H}_\rho + \tilde{H}_{\lambda g}H_\rho). \nonumber
	\end{eqnarray}
	Meanwhile,
	\begin{eqnarray*}
		\eqref{subkes1} &=& C_w\sum_{a\ge 2}\sum_{k\ge 1} \frac{|\rho|_{\lambda,a}}{((a-1)!)^{2s}}\frac{|\gamma g|_{\lambda,k}}{(k!)^{2s}}(a+k)^{s-1}\left(\frac{k!(a-1)!}{(a+k-1)!}\right)^s \\
		&\le& C_w\sum_{a\ge 2}\sum_{k\ge 1} (a+1)^{s-1}\frac{|\rho|_{\lambda,a}}{((a-1)!)^{2s}}\frac{|\gamma g|_{\lambda,k}}{(k!)^{2s}} k^{s-1} \left(\frac{k!(a-1)!}{(a+k-1)!}\right)^s\\
		&\le& C_w\sum_{a\ge 2}\sum_{k\ge 1} (a+1)^{s-1}\frac{|\rho|_{\lambda,a}}{((a-1)!)^{2s}}\frac{|\gamma g|_{\lambda,k}}{(k!)^{2s}}\left(\frac{kk!(a-1)!}{(a+k-1)!}\right)^s.
	\end{eqnarray*}
	For $ a\ge2 $ and $ k\ge 1 $, we have
	\[
	\frac{kk!(a-1)!}{(a+k-1)!} = \frac{(a-1)!}{(a-1)!}\cdot1\cdot\frac{2\cdots k\cdot k}{a\cdot(a+1)\cdots(a+k-1)}\le 1,
	\]
	thus 
	\begin{eqnarray*}
		\eqref{subkes1} &\le& C_w\sum_{a\ge 2}\sum_{k\ge 1} (a+1)^{s-1}\frac{|\rho|_{\lambda,a}}{((a-1)!)^{2s}}\frac{|\gamma g|_{\lambda,k}}{(k!)^{2s}} \\
		&\le& C_wH_{\lambda g}\tilde{H}_\rho.
	\end{eqnarray*}
	We therefore obtain the estimate 
	\[
	\eqref{kes3} \le 2C_wH_{\lambda g}\tilde{H}_\rho +  C_w\tilde{H}_{\lambda g}H_\rho \le 2C_wC_\gamma H_g\tilde{H}_\rho+C_wC_\gamma \tilde{H}_g H_\rho.
	\]
	 Estimates \eqref{kes1} and \eqref{kes2} together with \eqref{kes3} imply that
	\[
	\partial_tH_g\le(\lambda_0-K+(4\cdot9^s+C_\gamma)C_wH_\rho)\tilde{H}_g+(18^s+2C_\gamma)C_wH_g\tilde{H}_\rho.
	\]
	For $ K\ge \lambda_0+(4\cdot9^s+C_\gamma)C_wM $, we have the estimate 
	\[
	H_g(t)\le H_g(0)e^{(18^s+2C_\gamma)C_w\tilde{M}}=M.
	\]
	We can thus conclude that
	\[
	(K-\lambda_0-(4\cdot9^s+C_\gamma)M)\tilde{H}_g\le (18^s+2C_\gamma)C_wM\tilde{H}_\rho-\partial_tH_g.
	\]
	Integrating from $ 0 $ to $ T $ yields
	\begin{eqnarray*}
		(K-\lambda_0-(4\cdot9^s+C_\gamma)C_wM)\int_{0}^{T}\tilde{H}_gdt &\le& (18^s+2C_\gamma)C_wM\tilde{M}-H_g(T)+H_g(0)\\
		&\le&(18^s+2C_\gamma)C_wM\tilde{M}+2M.
	\end{eqnarray*}
	Thus for
	\[
	K\ge \frac{(18^s+2C_\gamma)C_wM\tilde{M}+2M.}{\tilde{M}}+\lambda_0+(4\cdot9^s+C_\gamma)C_wM,
	\]
	we obtain that
	\[
	\int_{0}^{T}\tilde{H}_g dt \le \tilde{M}.
	\]
	We have thus shown that $ g \in \mathcal{H}_{T,M}$ and that the desired $ T,K $ satisfy
	\[
	K\ge \frac{(18^s+2C_\gamma)C_wM\tilde{M}+2M.}{\tilde{M}}+\lambda_0+(4\cdot9^s+C_\gamma)C_wM, \quad0\le T\le\frac{\lambda_0}{1+K}.
	\] 
This remark completes the proof. 
\end{proof}
\subsection{Proof of Theorem \ref{mainthem}}
 Recall that $ \alpha(v) \ge 0$ is chosen such that $ \int_\mathbb{R}\alpha dv\le 1 $ and that $ \gamma(v)=\alpha'(v)/\alpha(v) $ satisfies
\[
	C_\gamma = \sum_{k\ge0} \sum_{0\le l\le k} C_k^l\lambda_0^l\sum_{0\le n\le k}\left(|\gamma_n|_{L^\infty_v}\right)<\infty.
\]
\begin{rem}
Notice that functions $ \alpha $ like above do exist. Indeed, take for instance $ \alpha(v)=C{e^{-v^2}}$, where $ 0\le C \le 1/\sqrt{\pi} $.
\end{rem}
Given a distribution function $ f(t,x,v) $, set $ \rho(t,x) = \int_{\mathbb{R}}\alpha fdv $ and define a map $ \Phi $ by $ \Phi(f)(t,x,v) := g(t,x,v)$, where $ g $ is the solution to 
\begin{equation*}
	\partial_t g + v\partial_x g + (W* \partial_x\rho) (\partial_v g+\gamma g) = 0
\end{equation*}
with $ g(0,x,v)=f_i(x,v) $.
We will first show that there exist $ M,T $ such that $ \Phi $ maps $ \mathcal{H}_{T,M} $ into itself and then prove that $ \Phi $ is a contraction on $ \mathcal{H}_{T,M} $ to finish the proof.

First, observe that, thanks to Lemma \ref{lemma1} and Lemma \ref{lemma2}, the map $ \Phi $ is well defined. Next, let $ f(t,x,v) , \tilde{f}(t,x,v) \in\mathcal{H}_{T,M}$ with $ f(0,x,v)=\tilde{f}(0,x,v)=f_i(x,v) $, $ \rho=\int_\mathbb{R} fdv,\tilde{\rho}=\int_\mathbb{R}\tilde{f}dv $ and $ g=\Phi(f), \tilde{g}=\Phi(\tilde{f}) $. It is easy to see $ g-\tilde{g} $ satisfies
	\[
		\partial_t (g-\tilde{g}) + v \partial_x(g-\tilde{g}) + (W*\partial_x\rho)(\partial_v(g-\tilde{g})+\gamma(g-\tilde{g}) +(W*\partial_x(\rho-\tilde{\rho}))(\partial_v\tilde{g}+\gamma\tilde{g}) = 0.
	\]
	By Lemma \ref{lemma1}, 
	\begin{eqnarray*}
		\frac{d}{dt}|g-\tilde{g}|_{\lambda,a} &\le& a^s|g-\tilde{g}|_{\lambda,a} + (\lambda_0-1-K)|g-\tilde{g}|_{\lambda,a+1} \\
		&&+C_wD^{a,s}(|\rho|_{\lambda,1}|g-\tilde{g}|_{\lambda,1}+|\rho|_{\gamma,1}|\gamma (g-\tilde{g})|_{\gamma})\\
		&&+C_wD^{a,s}(|\rho-\tilde{\rho}|_{\lambda,1}|\tilde{g}|_{\lambda,1}+|\rho-\tilde{\rho}|_{\lambda,1}|\gamma \tilde{g}|_\gamma).
	\end{eqnarray*}
Estimates of \eqref{kes1}, \eqref{kes2}, and \eqref{kes3} of  Lemma \ref{lemma2} imply that
	\begin{eqnarray*}
	\partial_t H_{g-\tilde{g}} &\le& (\lambda_0-K)\tilde{H}_{g-\tilde{g}}+(18^s+2C_\gamma)C_wH_{g-\tilde{g}}\tilde{H}_\rho+(4\cdot9^s+C_\gamma)H_\rho\tilde{H}_{g-\tilde{g}}\\
	&&+(18^s+2C_\gamma)C_wH_{\tilde{g}}\tilde{H}_{\rho-\tilde{\rho}} +(4\cdot9^s+C_\gamma)H_{\rho-\tilde{\rho}}\tilde{H}_{\tilde{g}}\\
	&=& (18^s+2C_\gamma)C_wH_{g-\tilde{g}}\tilde{H}_\rho+(\lambda_0-K+(4\cdot9^s+C_\gamma)H_\rho)\tilde{H}_{g-\tilde{g}}\\
	&&+(18^s+2C_\gamma)C_wH_{\tilde{g}}\tilde{H}_{\rho-\tilde{\rho}} +(4\cdot9^s+C_\gamma)H_{\rho-\tilde{\rho}}\tilde{H}_{\tilde{g}}\\
		 &\le&(18^s+2C_\gamma)C_wH_{g-\tilde{g}}\tilde{H}_f+(\lambda_0-K+(4\cdot9^s+C_\gamma)H_f)\tilde{H}_{g-\tilde{g}}\\
		&&+(18^s+2C_\gamma)C_wH_{\tilde{g}}\tilde{H}_{f-\tilde{f}} +(4\cdot9^s+C_\gamma)H_{f-\tilde{f}}\tilde{H}_{\tilde{g}}.
	\end{eqnarray*}
	If $ \bar{K}=K-\lambda_0-(4\cdot9^s+C_\gamma)M $, we can write that
	\begin{eqnarray*}
		&& e^{-\int_0^tC_w(18^s+2C_\gamma)\tilde{H}_f(s)ds}\bar{K}\tilde{H}_{g-\tilde{g}}+\frac{d}{dt}\left(e^{-\int_0^tC_w(18^s+2C_\gamma)\tilde{H}_f(s)ds}H_{g-\tilde{g}}\right) \\
		&\le& C_wM(18^s+2C_\gamma)\tilde{H}_{f-\tilde{f}}+(4\cdot9^s+C_\gamma)H_{f-\tilde{f}}\tilde{H}_{\tilde{g}}.
	\end{eqnarray*}
	Integrating over time, for $0\le t\le T $, yields
	\begin{eqnarray*}
		&& \frac{h_0}{M}\left(\bar{K}\int_0^t\tilde{H}_{g-\tilde{g}}(s)ds+H_{g-\tilde{g}}(t)\right)\\
		&\le& C_wM(18^s+2C_\gamma)\int_0^t\tilde{H}_{f-\tilde{f}}(s)ds+(4\cdot9^s+C_\gamma)\tilde{M}\sup_{0\le s\le t}H_{f-\tilde{f}}(s).
	\end{eqnarray*}
	Consequently
	\begin{eqnarray*}
		\frac{h_0}{M}\min{(1,\bar{K})}\|g-\tilde{g}\|_t \le \max{(C_wM(18^s+2C_\gamma),(4\cdot9^s+C_\gamma)\tilde{M})}\|f-\tilde{f}\|_t.
	\end{eqnarray*}
	Thus, taking $ t=T, $ we obtain
	\[
	\|g-\tilde{g}\|_T\le\frac{M}{h_0}\frac{\max((C_wM(18^s+2C_\gamma),(4\cdot9^s+C_\gamma)\tilde{M}))}{\min((1,K-\lambda_0-(4\cdot9^s+C_\gamma)M))}\|f-\tilde{f}\|_T.
	\]
	Finally, in order to make $\Phi$ a contraction, $ K,M,T $ should satisfy the conditions in Lemma \ref{lemma2},
	\begin{equation}\label{condition1}
	K\ge \frac{(18^s+2C_\gamma)C_wM\tilde{M}+2M}{\tilde{M}}+\lambda_0+(4\cdot9^s+C_\gamma)M,\quad 0\le T\le \frac{\lambda_0}{1+K},
	\end{equation}
	\begin{equation}\label{condition2}
	\frac{M}{h_0}\frac{\max((C_wM(18^s+2C_\gamma),(4\cdot9^s+C_\gamma)\tilde{M}))}{\min((1,K-\lambda_0-(4\cdot9^s+C_\gamma)M))}< 1, \ \ M> h_0.
	\end{equation}
These remarks complete the proof. \hfill $\square$

\section{Linear Stability}

In this section we will derive some Penrose-type conditions
for linear stability around homogeneous solutions in the sense of Landau damping. Let us start by recalling from Proposition \ref{eq_geo}  that the Vlasov-Poisson system restricted to a geodesic of $ \mathbb{M}^2 $ is given by the equations
\begin{equation}\label{vlasov_H}
\left\{
\begin{aligned}
& \partial_t f + v\partial_x f+ F(t,x)\partial_v f = 0,\\
& F(t,x) = W*\partial_x \rho,\\
& \rho(t,x) = \int_\mathbb{R} f(t,x,v)dv,
\end{aligned}
\right.
\end{equation}
with $ (x,v)\in I_{\mathbb M^2}\times{\mathbb R}$ and potential $W(x) = \frac{1}{2\pi}\log|\ctn(\frac{x}{2})| $, where $ * $ denotes the convolution in $ x $.

We want to find conditions for linear stability around the homogeneous solution $ f^0({v}) $. Denoting by $ h(t,x,{v}) $ the fluctuation of $f$ about $f^0$, we can assume in the linearization process that the nonlinear term $ F\frac{\partial h}{\partial {v}} $ is negligible. In doing so, the fluctuation $h$ becomes the solution to the system
\begin{equation}\label{linear}
	\left\{
	\begin{aligned}
		&\frac{\partial h}{\partial t} + {v}\frac{\partial h}{\partial x} + F(t,x) \frac{\partial f^0}{\partial {v}} = 0\\
		&F = W*\frac{\partial}{\partial x}\rho\\
		&\rho(t,x) = \int_\mathbb{R} h(t,\theta,{v})d{v}.
	\end{aligned}
	\right.
\end{equation}
In the periodic case, the linear stability around the homogeneous solution $ f^0({v})$ means that both the density and the force corresponding to the solutions of \eqref{linear} converge exponentially fast to the space-mean of the density or to zero, respectively. More precisely, we have the following result in the case of $\S^2$.
\begin{thm}\label{thm_S2}
	Assume that, in $\S^2$,
	\begin{itemize}
	\item the stationary solution $f^0=f^0({v}) $ of system \eqref{linear} and the initial perturbation $h_0=h_0(x,{v}) $ are analytic functions; 
	
	\item $ (f^0)'({v}) = O(1/|{v}|) $ for large enough values of $|{v}| $; 
	\item the Penrose stability condition takes place, i.e. 
	\begin{equation}\label{stability}
	\mbox{if}\ {\omega}\in\mathbb{R} \mbox{ is such that}\; \ (f^0)'({\omega})=0, \ {\rm then}\ \ \! p.v.\int_{-\infty}^\infty\frac{(f^0)'(v)}{v-\omega}dv\ \! > -1.
	\end{equation}
	\end{itemize}
	Then there exist positive constants $\delta$ and $C$, depending on the initial data, such that for $t\geq1$ we have
	$$
	\left\lvert\hspace{-0.4mm}\left\lvert\rho(t,x)-\int_{\mathbb R}\int_0^{2\pi} h_0(x,{v})dx d{v} \right\rvert\hspace{-0.4mm}\right\rvert_{C^r(0,2\pi)} \le Ce^{-\delta t} \ \ {\rm and}\ \
	\|F(t,x)\|_{C^r(0,2\pi)} \le Ce^{-\delta t},
	$$
 where $ \|u\|_{C^r(0,2\pi)} := \max_{0\le n\le r, \\
 0< x<2\pi} {|\partial^n_x u(x)|} $ and $ r \in \mathbb{N}^+$.
\end{thm}
To state a similar linear stability result in the case of a hyperbolic circle in $\H^2$, we need to redefine the norm of analytic functions in terms of the Fourier transform, see for example \cite{MV}. For any function $ f $, define this norm  by
\[
	\|f\|_{\mathcal{F}^\lambda} :=\int_\mathbb{R} e^{\lambda|\xi|}|\hat{f}(\xi)|d\xi,
\]
where $ \hat f $ stands for the Fourier transform of $f$. With this preparation, we can now state our  linear stability result in $\H^2$.
\begin{thm}\label{thm_H2}
	Assume that, in $\H^2$, 
	\begin{itemize}
	\item
	the stationary solution $f^0=f^0({v}) $ of system \eqref{linear} and the initial perturbation $h_0=h_0(x,{v}) $ are analytic functions; 
	\item
	$ (f^0)'({v}) = O(1/|{v}|) $ for large enough values of $|{v}| $; 
	\item the Penrose stability condition takes place, i.e. 
	\begin{equation}\label{stability_h}
	{\rm for \ every}\ {\omega}\in\mathbb{R}, \ (f^0)'({\omega})=0 \ {\rm implies\ that}\ \left(p.v.\int_{-\infty}^\infty\frac{(f^0)'(v)}{v-{\omega}}dv\right) > -\frac{4}{\pi}.
	\end{equation} 
	\end{itemize}
	Then there exist constants $\lambda',C>0$, which depend on the initial data, such that for $t$ large we have
	$$
	\|\rho(t,\cdot)\|_{\mathcal{F}^{\lambda'}}\le \frac{2C}{\lambda't} \ \ {\rm and}\ \
	\|F(t,\cdot)\|_{\mathcal{F}^{\lambda'}} \le \frac{C}{\lambda't}.
	$$
\end{thm}
The proofs of Theorem \ref{thm_S2} and Theorem \ref{thm_H2}  rely on the following lemma about the decay of solutions of Volterra-type equations (see  \cite{VillaniLectures} for its proof).

\begin{lem}\label{lemma}
	Assume that $ \phi $ solves the equation
	\[
	\phi(t)=a(t)+\int_0^tK(t-\tau)\phi(\tau)d\tau,
	\]
	and there are constants $c, \alpha, C_0, \lambda, \lambda_0,\Lambda>0$, such that the function $a$ and the kernel $K$ satisfy the conditions
	\begin{itemize}
		\item[(i)]$|K(t)|\le C_0e^{-\lambda_0t}$;
		\item[(ii)]$|K^L(\xi)-1|\ge c$ for $0\le Re\ \!\xi \le \Lambda$;
		\item[(iii)]$|a(t)|\le \alpha e^{-\lambda t}$,
	\end{itemize}
where $K^L(\xi)$ is the complex Laplace transform defined by
	\[
	K^L(\xi)=\int_0^\infty e^{\bar{\xi}t}K(t)dt\  \text{for $\xi\in\mathbb{C}$}.
	\]
	Then for any positive $\lambda'\le\min(\lambda,\lambda_0,\Lambda)$, we have the inequality
	\[
	|\phi(t)|\le Ce^{-\lambda't},
	\]
	with
	\[
	C=\alpha+\frac{C_0\alpha}{2\sqrt{(\lambda_0-\lambda')(\lambda-\lambda')}}.
	\] 
\end{lem}


\subsection{Proof of Theorem \ref{thm_S2}}

First, observe that the conservation of mass for the solution $h$ of \eqref{linear} is equivalent to the conservation of the zero mode of the density function $\rho$,
\[
	\hat{\rho}(t,0) = \tilde{h}_0(0,0).
	\]
In the following, we estimate each mode $ \hat{\rho}(t,k) $ first, and then use it to show the convergence of $ \rho(t,x) $ and $ F(t,x) $ with the exponential rate. The Penrose type stability condition comes naturally to satisfy the second hypothesis of Lemma \ref{lemma}.

In order to deal with higher modes, we first solve system \eqref{linear} using the method of characteristics and a Duhamel-type formula. Setting $S(t,x,v):=F(t,x)\partial_{v} f^0(v)$,  we have
	\begin{equation}\label{sol_h}
		h(t,x,{v})=h_0(x-{v} t, {v})-\int_0^tS(\tau,x-{v}(t-\tau),{v}) d\tau.
	\end{equation}
Taking the Fourier transform $\bar h$ of $h$ both in $x$ and ${v}$ gives 

	\begin{eqnarray}\label{F(sol)_temp}
		\bar{h}(t,k,\eta) 
		&=& \bar{h}_0(k,\eta+kt)-\int_0^t \bar{S}(\tau,k,\eta+k(t-\tau)) d\tau.
	\end{eqnarray}
With the help of the identity
\[
	\log(\cos x) = \sum_{n=1}^{+\infty}(-1)^{n+1}\frac{\cos(2nx)}{n}-\log 2,
\]
we calculate that
	 \[ \widehat{W}(k)=\left\{
	 \begin{aligned}
	 & \frac{1}{|k|},& \text{$ k $ odd}\\
	 & 0,& \text{$k$ even}.
	 \end{aligned}
	 \right.
	 \]
Now, since $ S $ has the structure of separated variables, 
	 we have
	\begin{equation}\label{F(S)}
		\bar{S}(\tau,k,\eta) = \widehat{F}(\tau,k)\widehat{\partial_{v} f^0}(\eta) 
		=i\eta \widehat{F}(\tau,k)\hat{f}^0(\eta) \nonumber\\
		=\left\{
		\begin{array}{ll}
			0,&\text{k even}\\
			-\eta\frac{k}{|k|}\hat{\rho}(\tau,k)\hat{f}^0(\eta),&\text{k odd}.
		\end{array}
		\right. 
	\end{equation}
	Substituting this expression into \eqref{F(sol)_temp}, we obtain
	\begin{equation}\label{F(sol)}
		\bar{h}(t,k,\eta) = \left\{
		\begin{aligned}
			&\bar{h}_0(k,\eta+kt),&\text{k even} \\
			&\bar{h}_0(k,\eta+kt)+\int_0^t\frac{k}{|k|}(\eta+k(t-\tau))\hat{\rho}(\tau,k)\hat{f}^0(\eta+k(t-\tau))d\tau,&\text{k odd}.
		\end{aligned}
		\right.
	\end{equation}
Finally, the choice $\eta=0$ gives
	\begin{eqnarray}\label{rho}
		\hat{\rho}(t,k) 
		&=& \bar{h}_0(k,kt)+\int_0^t K(t-\tau)\hat{\rho}(\tau,k)d\tau,
	\end{eqnarray}
	where
	\[
	K(t,k) = \left\{
	\begin{array}{ll}
	0,&\text{k even}\\
	|k|t\hat{f}^0(kt),&\text{k odd}.
	\end{array}
	\right.
	\]
Clearly, for any given $k\ne 0$,  \eqref{rho} is a Volterra type equation. Since $f^0$ and $h_0$ are analytic, for large $ t $ we have
	$$ 
	|K(t,k)| = O(e^{-\lambda_0|k|t})\ {\rm and}\ |\bar{h}_0(k,kt)|=O(e^{-\lambda_1|k|t}).
	$$
Assuming that the last hypothesis is also satisfied, Lemma \ref{lemma} implies that
	\[
	|\hat{\rho}(t,k)| = O(e^{-\lambda'|k|t}),
	\]
with $ \lambda' < \min\{\lambda_0, \lambda_1\}$. Now take $ r\in\mathbb{N}^+ $ and $ t \geq1$. Then there exists a positive constant $ C $, which depends only on the initial data, such that
	\begin{eqnarray*}
		\left|\partial^r_x\left(\rho(t,k)-\int_{\mathbb R}\int_0^{2\pi} h_0(x,v)dx dv\right)\right| &\le& \sum_{k\ne 0} |k|^r|\hat{\rho}(t,k)| \le C\sum_{k\ne 0} |k|^re^{-\lambda'|k|t}.
	\end{eqnarray*}
Choosing a constant $ 0<\delta < \lambda'$ such that for all $ k\ne 0 $ and $ r\in\mathbb{N}^+ $ we have
	\[
	e^{-(\lambda'-\delta)|k|t}\le |k|^{-r-2},
	\]
we can obtain the estimate 
	$$
	\left|\partial^r_x\left(\rho(t,k)-\int_{\mathbb R}\int_0^{2\pi} h_0(x,v)dx dv\right)\right| 
	\le C \sum_{k\ne0} |k|^{-2}e^{-\delta|k|t}
	\le \frac{CC^*}{2}e^{-\delta t},
	$$
	where $ C^* = \sum_{k=1}^{\infty} k^{-2}$. Similarly for the force, we have 
	\[
	\partial^r_x F(t,\theta) = \sum_{k\text{ odd}} (ik)^ri\hat{\rho}e^{ik\theta},
	\]
and therefore, for $t\geq1$, we have
	\[
	\|F(t,\theta)\|_{C^r(\mathbb{T})} \le \frac{CC^*}{2}e^{-\delta t},
	\]
as desired. Finally, it only remains to check the second condition of Lemma \ref{lemma}. It basically means that when $Re\ \! \xi$ is located in a positive neighbourhood of $0$, the Laplace transform of the kernel $K^L(\xi,k)$ should stay away from $1$.
	
	When $k$ is even, $K^L(\xi,k)\equiv0$, so it's always away from $1$. When $k$ is odd, we evaluated the Laplace transform of $K$ in time at $\xi=(\lambda-i\omega)k$ and obtained
	\begin{eqnarray}\label{L(K)}
		K^L(\xi,k)&=&\int_0^\infty |k|t \hat{f}^0(kt)e^{(\lambda+i{\omega})tk}dt\nonumber\\
		&=& \int_{0}^{\infty}e^{(\lambda+i\omega)kt}|k|t\int_\mathbb{R}e^{-i\omega kt}f^0(v)dv \nonumber \\
		&=& \int_\mathbb{R}-i\frac{|k|}{k}(f^0)'(v)\int_0^\infty e^{(\lambda+i(\omega-v))kt}dtdv \nonumber\\
		&=& -\frac{1}{|k|}\int_\mathbb{R}\frac{(f^0)'(v)}{i\lambda+(v-\omega)}dv.
	\end{eqnarray}
	Moreover, if the stationary solution $f^0$ has the property that $(f^0)'(v)$ decays at least like $O(1/|v|)$, then  \eqref{L(K)} implies that
	\begin{equation}\label{L(K)decay}
		|K^L(\xi,k)|\le C\left|\frac{1}{k}\int_{\mathbb{R}}\frac{dv}{v(i\lambda-{\omega}+v)}\right|.
	\end{equation}
	Let $L(\varepsilon)$ be upper half circle centered at $0$ with radius $\varepsilon$. Some simple computations show that
	\begin{eqnarray*}
		\frac{1}{k}\int_{\mathbb{R}}\frac{1}{v}\frac{1}{i\lambda-\omega+v}dv &=& \frac{1}{k}\lim_{\varepsilon\to 0} \int_{-\infty}^{-\varepsilon}+\int_\varepsilon^{+\infty}+\int_{L(\varepsilon)}\frac{1}{v}\frac{1}{i\lambda-\omega+v}dv\\
		&=&\frac{\pi i}{k(i\lambda-\omega)} + \frac{1}{k}\lim_{\varepsilon\to 0} \int_\pi^0 \frac{1}{\varepsilon e^{i\phi}} \frac{1}{i\lambda-\omega+\varepsilon e^{i\phi}} d(\varepsilon e^{i\phi})\quad\\
		&=& (\text{with $v=\varepsilon e^{i\phi}$})\ \ \frac{2\pi i}{k(i\lambda-\omega)},
	\end{eqnarray*}
	which means that \eqref{L(K)} decays at least like $O(1/|\omega|)$ as $v \to \infty$, uniformly for $\lambda\in[0,\lambda_0]$, so it is enough to consider only the case when $|v|$ is bounded. Hence, assume that $|\omega|\le \Omega$. If \eqref{L(K)} does not go to $1$ when $\lambda\to0^+$, by continuity there exists $\Lambda>0$ such that \eqref{L(K)} is away from $1$ in the domain $\{|\omega|\le\Omega, 0\le\lambda\le\Lambda\}$. Thus, we could only focus on the limit $\lambda\to 0^+$. In order to compute this limit, we introduce the Plemelj formula (see \cite{MV}),
	\begin{equation}\label{Plemelj}
		\lim_{y\to0^+}\int_{-\infty}^\infty\frac{f(x)}{x-x_0+iy}dx=p.v.\int_{-\infty}^\infty\frac{f(x)}{x-x_0}dx-i\pi f(x_0).
	\end{equation}
	Applying \eqref{Plemelj} to \eqref{L(K)}, we have that
	\begin{equation}\label{limit}
		\lim_{\lambda\to0^+} K^L((\lambda-i\omega)k,k) = -\frac{1}{|k|}\left(p.v.\int_{-\infty}^\infty\frac{(f^0)'(v)}{v-\omega}dv-i\pi (f^0)'({\omega})\right). 
	\end{equation}
	We further need to find conditions such that \eqref{limit} does not approach $1$. First, if the imaginary part of \eqref{limit} stays away from $0$, then everything is fine. However, the imaginary part goes to $0$ only if $k\to\infty$, in which case the real part also goes to $0$, or if $(f^0)'(v)\to 0$. So we only need to consider the case when $(f^0)'(\omega)$ approaches $0$. Hence, we need to require that
	\begin{equation}\label{criterion_temp}
		{\rm for \ every}\ {\omega}\in\mathbb{R}, \ (f^0)'(\omega)=0 \ {\rm implies\ that}\ -\frac{1}{|k|}\left(p.v.\int_{-\infty}^\infty\frac{(f^0)'(v)}{v-{\omega}}dv\right) \neq 1,
	\end{equation}
a hypothesis that leads to the Penrose stability condition \eqref{stability}. \hfill $\square$

\subsection{Proof of Theorem \ref{thm_H2}}
Arguing as in the proof of the previous theorem, we just observe that $ \widehat{W}(\xi)=\frac{1}{2\xi}\tanh(\frac{\xi\pi}{2})$ (see \cite{Fritz}) and that
\begin{equation}\label{rho_hyper}
	\hat{\rho}(t,\xi) = \bar{h}_0(\xi,\xi t) + \int_0^tK(k-\tau)\hat{\rho}(\tau, \xi)d\tau,
\end{equation}
where $ K(t,\xi)=\frac{\xi t}{2}\tanh(\frac{\xi \pi}{2}) \hat{f}^0(\xi t)$. Moreover, the force is 
\begin{eqnarray*}
	\widehat{F}(t,\xi)&=&i\xi\widehat{W}(\xi)\hat{\rho} (\xi)= \frac{i}{2}\tanh(\frac{\xi\pi}{2})\hat{\rho}(\xi).
\end{eqnarray*}
Computing the Laplace transform $ K^L(\zeta,\xi) $ at $ \zeta=(\lambda-i\omega)\xi $, we get
\begin{eqnarray}
	K^L(\zeta, \xi)&=&\int_0^\infty e^{(\lambda+i\omega)\xi t}\frac{\xi t}{2}\tanh(\frac{\xi \pi}{2}) \int_\mathbb{R}e^{-i\xi t v} f^0(v)dv dt \nonumber\\
	&=& \int_0^\infty e^{(\lambda+i\omega)\xi t}\frac{\xi t}{2}\tanh(\frac{\xi \pi}{2}) \frac{1}{i\xi t} \int_\mathbb{R} (f^0)'(v)e^{-i\xi v t} dv dt \nonumber\\
	&=& \int_\mathbb{R} -\frac{i}{2} \tanh(\frac{\xi \pi}{2}) (f^0)'(v)\int_0^\infty e^{(\lambda+i(\omega-v))\xi t} dt dv \nonumber\\
	&=& -\frac{1}{2\xi}\tanh(\frac{\xi\pi}{2})\int_\mathbb{R}\frac{(f^0)'(v)}{i\lambda+(v-\omega)}dv.
\end{eqnarray}
Going though the same argument as we did in the previous section, we only need to consider the case when $ (f^0)'(\omega) $ approaches $ 0 $. Hence, we have to assume that 
\begin{equation}
	{\rm for \ every}\ {\omega}\in\mathbb{R}, \ (f^0)'({\omega})=0 \ {\rm implies\ that}\ -\frac{1}{2\xi}\tanh(\frac{\xi\pi}{2})\left(p.v.\int_{-\infty}^\infty\frac{(f^0)'(v)}{v-{\omega}}dv\right) \neq 1,
\end{equation}
which is equivalent to
\begin{equation}
	{\rm for \ every}\ {\omega}\in\mathbb{R}, \ (f^0)'({\omega})=0 \ {\rm implies\ that}\ \left(p.v.\int_{-\infty}^\infty\frac{(f^0)'(v)}{v-{\omega}}dv\right) \neq -\frac{2\xi}{\tanh(\frac{\xi\pi}{2})}.
\end{equation}
Since $ -\frac{2\xi}{\tanh(\frac{\xi\pi}{2})}\le -\frac{4}{\pi} $, we obtain the Penrose stability condition \eqref{stability_h}.
Then the application of Lemma \ref{lemma}  provides us the exponential decay of the density
\[
	|\hat{\rho}(t,\xi)|\le Ce^{-\lambda'|\xi|t},
\]
for some constants $C$ and $\lambda'$. From this, we can conclude that
\begin{equation*}
	\|\rho(t,\cdot)\|_{\mathcal{F}^{\lambda'}} = \int_\mathbb{R} e^{\lambda'|\xi|}|\hat{\rho}(t,\xi)|d\xi \le C\int_\mathbb{R}e^{-\lambda't|\xi|}d\xi =\frac{2C}{\lambda'(t)}
\end{equation*}
and 
\begin{equation*}
	\|F(t,\cdot)\|_{\mathcal{F}^{\lambda'}} = \int_\mathbb{R} \frac{1}{2}e^{\lambda'|\xi|}\tanh(\frac{|\xi|\pi}{2})|\hat{\rho}(t,\xi)|d\xi \nonumber \le \frac{C}{2}\int_\mathbb{R}e^{\lambda'|\xi|}e^{-\lambda'|\xi|t}d\xi =\frac{C}{\lambda't},
\end{equation*}
a remark that completes the proof. \hfill $\square$

\section{Acknowledgements}
F.\ Diacu and S.\ Ibrahim were supported in part by Discovery Grants from NSERC of Canada, C.\ Lind received a CGS fellowship from the same institution, and S.\ Shen and C.\ Lind were partially supported by University of Victoria fellowships. The authors are grateful to R. Illner for several discussions on this topic and for his excellent suggestions. S.\ Ibrahim also recognizes the support from the ``Thematic Semester on Variational Problems in Physics, Economics and Geometry'' held at the Fields Institute, Toronto, Canada. He would also like to thank N.\ Masmoudi, W.\ Gangbo, and F.\ Otto for fruitful discussions on the subject. The figures were produced using the GeoGebra4 software,  available in the public domain at http://www.geogebra.org.

\bibliography{MyBib}
	\bibliographystyle{acm}
	
\end{document}